\documentclass{article}
\usepackage[utf8]{inputenc}
\usepackage{amsfonts}
\usepackage{hyperref}
\usepackage{enumitem}
\usepackage{amsmath}
\usepackage{amsthm}
\usepackage[capitalize,nameinlink]{cleveref}
\usepackage{comment}
\usepackage{mathrsfs}
\usepackage{mathabx}
\usepackage{tikz}
\usepackage{todonotes}
\usetikzlibrary{positioning}
\usepackage{caption}
\newtheorem{theorem}{Theorem}[section]
\usepackage[a4paper, margin=1in]{geometry}
\newtheorem{corollary}[theorem]{Corollary}
\newtheorem{lemma}[theorem]{Lemma}

\newtheorem{proposition}[theorem]{Proposition}

\newtheorem{definition}[theorem]{Definition}
\newtheorem{example}[theorem]{Example}
\newtheorem{claim}[theorem]{Claim}
\usepackage{stackengine} 
\newcommand\oast{\stackMath\mathbin{\stackinset{c}{0ex}{c}{0ex}{\ast}{\bigcirc}}}
\DeclareMathOperator*{\esssup}{ess\,sup}

\newcommand{\Addresses}{{
    \bigskip
    \footnotesize

    Franz Luef, \textsc{Department of Mathematical Sciences, Norwegian University of Science and Technology, 7034
    Trondheim, Norway}\par\nopagebreak
    \textit{E-mail address}: \texttt{franz.luef@ntnu.no}
    
    \medskip
    
    Henry McNulty, \textsc{Cognite AS, 1366 Lysaker, Norway}\par\nopagebreak
    \textsc{Department of Mathematical Sciences, Norwegian University of Science and Technology, 7491
    Trondheim, Norway}\par\nopagebreak
    \textit{E-mail address}: \texttt{henry.mcnulty@cognite.com}

}}

\title{Quantum Time-Frequency Analysis and Pseudodifferential Operators}
\author{
  Franz Luef \and Henry McNulty 
  }
\begin{document}

\maketitle
\begin{abstract}
    We introduce Quantum Time--Frequency Analysis, which expands the approach of Quantum Harmonic Analysis to include modulations of operators in addition to translations. This is done by a projective representation of double-phase space, and we consider the associated matrix coefficients and integrated representation. This leads to the polarised Cohen's class, which is an isomorphism from Hilbert-Schmidt operators to a reproducing kernel Hilbert space, and has orthogonality relations similar to many objects in classical time--frequency analysis. By considering a class of windows for the polarised Cohen's class that is smaller than the class of Hilbert-Schmidt operators, then we find spaces of modulation spaces of operators, and we consider the properties of these spaces, including discretisation results and mapping properties between function modulation spaces. We also compare modulation spaces of operators to known symbol classes for pseudodifferential operators. In many cases, using rank--one examples of operators, we recover familiar objects and results from classical time--frequency analysis and the theory of pseudodifferential operators.
\end{abstract}

\section{Introduction}
Quantum Harmonic Analysis (QHA), introduced by Werner in 1984 \cite{werner84}, extends the tools of harmonic analysis, convolutions and Fourier transforms, to operators and functions on phase space, where the translation of an operator $S$ is defined by $\alpha_z(S)=\pi(z)S\pi(z)^*$ and $\pi(z)$ denotes the Schr\"odinger representation of $\mathbb{R}^{2d}$. 

Recently these tools have been recognised as valuable instruments to understand objects in time-frequency analysis, and have motivated many new developments in the field \cite{skrett18} \cite{skrett19} \cite{skrett20}, as well as generalisations to other locally compact group representations \cite{berge22} \cite{halv23} \cite{fulsche23}, and applications to the analysis of functional data sets \cite{dorf21} \cite{dorf22}. As one finds in classical harmonic analysis, the translations of operators $\alpha_z$ in QHA form a commutative representation of phase space, and many questions concerning operators may be reduced to questions on their Weyl symbols.
\paragraph{}
Motivated by the efficacy of time-frequency analysis of functions, in this work we extend our perspective to that of Quantum Time-Frequency Analysis. While classical time-frequency analysis occurs on phase space, the time-frequency analysis of operators (themselves identified via their Weyl symbol with phase space) requires lifting to double phase space. To do so, we consider a projective representation of double-phase space acting on the space of Hilbert-Schmidt operators,
\begin{align*}
    \gamma_{w,z} (S) := \pi(z)S\pi(w)^*.
\end{align*}
which amounts to a combination of translation and modulation of operators, understood on the symbol level, motivating our nomenclature. While $\alpha_z$ shifts are related to the \textit{covariant} symbol in the work \cite{Be75}, $\gamma_{w,z}$ shifts are related to the \textit{contravariant} symbol. In particular, we find:
\begin{theorem}
    Let $w=(w_1,w_2),z=(z_1,z_2)\in\mathbb{R}^{2d}$, and $S\in\mathcal{HS}$. Then
    \begin{align*}
        \sigma_{\gamma_{w,z}(S)} = e^{i\pi(w_1 + z_1)(w_2 - z_2)}\pi(U(w,z))\sigma_S,
    \end{align*}
    where
    \begin{align*}
        U(w,z) &= U(w_1,w_2,z_1,z_2) := \Big(\frac{w_1+z_1}{2},\frac{w_2+z_2}{2}, w_2 - z_2, z_1-w_1\Big).
    \end{align*}
\end{theorem}
By considering a rank-one operator, this can be used to show the form of an STFT of a Wigner distribution, used in many results about boundedness of operators on modulation spaces. We see that on the diagonal $w=z$, the time-frequency shift $\gamma_{z,z}$ reduces to a pure translation $\alpha_z$, and in this sense we find quantum harmonic analysis within quantum time-frequency analysis. With this representation, we consider the matrix coefficients and integrated representations, and define these as the Polarised Cohen's class and its adjoint respectively;
\begin{align*}
    Q_S T(w,z) &:= \langle T, \gamma_{w,z}(S)\rangle_{\mathcal{HS}} \\
    Q_S^* F &:= \int_{\mathbb{R}^{4d}} F(w,z) \gamma_{w,z}(S)\, dw\, dz.
\end{align*}
Note that $Q_S(f\otimes f)(z,z)$ is the Cohen class $Q_S(f)(z)$ of the operator $S$ that was introduced in \cite{skrett19}. We show that the polarised Cohen's class behaves in ways one would hope when constructing a time-frequency analysis object; it is isometric from the Hilbert-Schmidt operators into $L^2(\mathbb{R}^{4d})$, it satisfies an orthogonality relation mirroring Moyal's identity, and a resolution of the identity. In particular, it defines a reproducing kernel Hilbert space:
\begin{proposition}
    Let $S\in\mathcal{HS}$ be non-zero. Then For any $T\in\mathcal{HS}$, we have the identity
    \begin{align*}
        Q_S^* Q_S T = T.
    \end{align*}
    Furthermore, this gives rise to the reproducing formula
    \begin{align*}
        Q_S T(w,z) = \int_{\mathbb{R}^4d} Q_S T(w',z')k_{w,z}(w',z')\, dw'\, dz',
    \end{align*}
    where $k_{w,z}(w',z') = \langle \pi(z')S\pi(w')^*,\pi(z)S\pi(w)^*\rangle_{\mathcal{HS}}$.
\end{proposition}
We are naturally led to the notion of modulation spaces of operators, given by
\begin{align*}
    \mathcal{M}^{p,q}_m := \{T\in\mathfrak{S}^\prime: Q_{\varphi_0\otimes \varphi_0} T\in L^{p,q}_m(\mathbb{R}^{4d})\},
\end{align*}
where $\varphi_0$ is the normalised Gaussian and $\mathfrak{S}^\prime$ are those operators with Weyl symbol in the space of tempered distributions. Of particular interest is the subtle difference between spaces of operators with Weyl symbols in $M^{p,q}_m(\mathbb{R}^{2d})$ spaces, and the operators $\mathcal{M}^{p,q}_m$, and we consider inclusion results between these two classes of spaces. The operator space $\mathcal{M}^{p,q}_m$ is then characterised by the same reproducing property:
\begin{proposition}
    Given $T\in\mathfrak{S}^\prime$;
    \begin{align*}
        T\in\mathcal{M}^{p,q}_m \iff Q_{S_0} T \oast Q_{S_0} S_0 = Q_{S_0} T.
    \end{align*}
\end{proposition}
Here $\oast$ is the appropriate notion of a twisted convolution on double phase space, to be defined.
\paragraph{}
As a result of the operator translations being unitarily equivalent to a translation of the Weyl symbol, it is not possible to construct a discrete frame of translates on a lattice, that is to say, a frame of the form $\{\alpha_{\lambda}(S)\}_{\lambda\in\Lambda}$ for Hilbert-Schmidt operators \cite{skrett20b}. With the time-frequency shifts for operators, we construct frames of the form $\{\gamma_{\lambda,\mu}(S)\}_{(\lambda,\mu)\in\Lambda\times M}$. As one would expect, a privileged role is played by the Feichtinger operators $\mathcal{M}^1$ considered in \cite{feich22} \cite{Bast23}, and given a frame for $\mathcal{HS}$ with window in $\mathcal{M}^1$, the operator modulation spaces $\mathcal{M}^{p,q}_m$ are characterised by $\ell^{p,q}_m$ decay of the polarised Cohen's class on the lattice:
\begin{theorem}
    An operator $T\in \mathcal{M}^{\infty}$ is in the space $\mathcal{M}^{p,q}_m$ if and only if $C_S (T) \in \ell^{p,q}_m$ for some $S\in\mathcal{M}^1_v$ which generates a Gabor frame. 
\end{theorem}
This characterisation gives a reconstruction formula for all operators in $\mathcal{M}^{p,q}_m$ spaces
\begin{corollary}\label{atomdecompintro}
    Let $S,\Tilde{S}\in\mathcal{M}^1_v$ generate a dual pair of Gabor frames for $\Lambda\times M$, that is to say $E_{\Tilde{S},S} = I_{\mathcal{HS}}$. Then for $T\in\mathcal{M}^{p,q}_m$;
    \begin{align*}
        T = \sum_{\Lambda\times M} Q_S T(\lambda,\mu) \gamma_{\lambda,\mu}(\Tilde{S}),
    \end{align*}
    where the sum converges unconditionally for $p,q<\infty$, otherwise in the weak-$*$ topology.
\end{corollary}
Recognizing that this decomposition amounts to the composition of an analysis operator, the matrix $(Q_S T(\lambda,\mu))_{\lambda \mu}$, and a synthesis operator, we proceed to study the action of operators in $\mathcal{M}^{p,q}_m$ spaces as linear maps between modulation spaces. We find that:
\begin{proposition}
    For $1\leq p,q < \infty$, $\mathcal{M}^{p,q}\subset \Pi^q(M^{p'};M^q)$.
\end{proposition}
Here $\Pi^q(M^{p'};M^q)$ is the Banach space of $q$-summing operators between modulation spaces $M^{p'}(\mathbb{R}^d)$ and $M^q(\mathbb{R}^d)$ (cf. \cite{konig}). Finally we relate the spaces $\mathcal{M}^{p,q}_m$ to those operators with Weyl symbols in $M^{p,q}_m(\mathbb{R}^{2d})$, finding the following:
\begin{theorem}
    Given $1\leq p \leq q \leq \infty$, if $T\in\mathcal{M}^{q,p}$ then $\sigma_T \in M^{p,q}(\mathbb{R}^{2d})$. Conversely given \newline $1 \leq q \leq p \leq \infty$, if $\sigma_T \in M^{p,q}(\mathbb{R}^{2d})$ then $T\in\mathcal{M}^{q,p}$.
\end{theorem}
If we consider rank-one operators for $S$ and/or $T$, we recover familiar concepts from time-frequency analysis. Results on boundedness of operators with Weyl symbols in modulation spaces (cf. \cite{gro04}, \cite{cordero03}, \cite{cordero19}, \cite{goetz13}) rely on the following identities, aka as ``magic identities", for functions $f,g,\varphi \in L^2(\mathbb{R}^{d})$:
\begin{align*}
    V_{W(\varphi,\varphi)} W(f,g)(w,z) &= e^{-2\pi w_2 z_2} V_{\varphi} f(w_1-\tfrac{z_1}{2}, w_2+\tfrac{z_2}{2}) \overline{V_{\varphi} g(w_1+\tfrac{z_1}{2}, w_2-\tfrac{z_2}{2})} \\
    V_{V_{\varphi} \varphi} V_g f (w,z) &= e^{-2\pi w_2 z_2} V_{\varphi} f(-z_2, w_2+z_1) \overline{V_{\varphi} g(-w_1-z_2, z_1)}.
\end{align*}
In the framework of QTFA, these results are simply the relation between $\gamma_{w,z}(S)$ and the effect the representation has on the Weyl symbol and spreading function, respectively. Similarly, the decomposition \cref{atomdecompintro}, in the case of a rank-one window $S = g\otimes g$, becomes a decomposition of an operator $T$ with respect to the Gabor matrix
\begin{align*}
    \big( \langle T \pi(\mu)g, \pi(\lambda)g \rangle \big)_{\lambda,\mu}.
\end{align*}
Such decompositions have been considered in for example \cite{BaGr17}, and are used to show boundedness conditions in \cite{cordero19}. The rank--one case of Gabor frames for Hilbert Schmidt operators is considered in \cite{balasz2008}. The adjoint map $Q_S^*$ for rank--one $S$ maps a symbol to the corresponding bilocalisation operator introduced in \cite{BoGa24}.

In this sense QTFA can be seen to be a unifying framework for vaguely related concepts in the time-frequency analysis of operators.

\section{Preliminaries}
\subsection{Time Frequency Analysis}
The basic objects in time-frequency analysis are closely connected to the projective unitary representation of the Weyl-Heisenberg group which gives rise to time-frequency shifts on $L^2(\mathbb{R}^d)$. These shifts can be defined as the composition of the modulation operator $M_{\omega}:f(t)\mapsto e^{2\pi i \omega t}f(t)$, and the translation operator $T_x:f(t)\mapsto f(t-x)$, by the identity
\begin{align*}
    \pi(z) = M_{\omega}T_x
\end{align*}
where $z=(x,\omega)\in \mathbb{R}^{2d}$. The operator $\pi(z)$ is unitary on $L^2(\mathbb{R}^d)$, and satisfies the identities
\begin{align*}
    \pi(z)\pi(z') &= e^{-2\pi i \omega' x}\pi(z+z') \\
    \pi(z)^* &= e^{-2\pi i x \omega}\pi(-z).
\end{align*}
The Short-Time Fourier Transform (STFT) can then be defined for functions $f,g\in L^2(\mathbb{R}^d)$, by
\begin{align}
    V_g f(z) := \langle f, \pi(z) g\rangle_{L^2}.
\end{align}
The function $g$ is referred to as the window, or atom, and is commonly taken to be a function concentrated around the origin in both time and frequency, such as the (normalised) Gaussian $\varphi_0(t)=2^{d/4}e^{\pi t^2}$. For a normalised $g\in L^2(\mathbb{R}^2)$, the map $V_g : f\mapsto V_g f$ is isometric from $L^2(\mathbb{R}^d)$ to $L^2(\mathbb{R}^{2d})$, and in fact we have the even stronger result, known as Moyal's identity, that given functions $f_1,f_2,g_1,g_2\in L^2(\mathbb{R}^d)$;
\begin{align*}
    \langle V_{g_1} f_1, V_{g_2} f_2 \rangle_{L^2(\mathbb{R}^{2d})} = \langle f_1, f_2 \rangle_{L^2(\mathbb{R}^d)} \overline{\langle g_1, g_2 \rangle_{L^2(\mathbb{R}^d)}}.
\end{align*}
The adjoint $V_g^*:L^2(\mathbb{R}^{2d}) \to L^2(\mathbb{R}^d)$ is given by
\begin{align*}
    V_g^*(F) :=  \int_{\mathbb{R}^{2d}} F(z) \pi(z)g\, dz,
\end{align*}
where the integral can be understood in the weak sense, and gives the reconstruction formula 
\begin{align*}
    f = \int_{\mathbb{R}^{2d}} V_g f(z) \pi(z)g\, dz.
\end{align*}
If we restrict the window $g$ to the Schwartz space $\mathscr{S}$, we can define the STFT for all tempered distributions $\mathscr{S}^\prime$, as $\mathscr{S}$ is invariant under $\pi$ shifts. Doing so allows us to define the modulation spaces $M^{p,q}_m$ for $1\leq p,q\leq \infty$ and $v$-moderate weight $m$ of polynomial growth (see Chapter 11 of \cite{grochenigtfa} for details on weight functions and function spaces), the space
\begin{align*}
    M^{p,q}_m(\mathbb{R}^{d}) := \{f\in (M^1_v(\mathbb{R}^{d}))': V_{\varphi_0} f \in L^{p,q}_m(\mathbb{R}^{2d})\},
\end{align*}
with the norm $\|f\|_{M^{p,q}_m}=\|V_{\varphi_0} f\|_{L^{p,q}_m}$. Given any $g\in M^1_v(\mathbb{R}^{d})$, the space $\{f\in (M^1_v(\mathbb{R}^{d}))': V_{g} f \in L^{p,q}_m(\mathbb{R}^{2d})\}$ is equal to the space $M^{p,q}_m(\mathbb{R}^{d})$, and the associated norms are equivalent. From the properties of the STFT above, it follows that $M^{2,2}(\mathbb{R}^{d})=L^2(\mathbb{R}^{d})$. 
As a result of the uniform continuity of the STFT, the modulation spaces have a useful description in terms of local smoothness and global decay. For our purposes, we define the Wiener amalgam space $W(L^{p,g}_m(\mathbb{R}^{d}))$ as the space of functions $F$ such that
    \begin{align*}
        \|F\|_{W(L^{p,q}_m(\mathbb{R}^{d}))} := \Big\|\esssup_{z\in T_{(k,l)}\Omega} F(z) \Big\|_{\ell^{p,q}_m} < \infty,
    \end{align*}
where $\Omega$ is the fundamental domain $[0,1]^d$. For a function $g$ in the modulation space $M^1_v(\mathbb{R}^d)$, $V_g g$ is then contained in $W(L^1_v(\mathbb{R}^{2d}))$, and we will make use of the following convolution relation for this space:
\begin{lemma}\label{weineramalgconv}
    Given $g\in W(L^1_v(\mathbb{R}^{2d}))$ and $f\in L^{p,q}_m(\mathbb{R}^{2d})$ continuous functions, where $m$ is a $v$-moderate weight, we have
    \begin{align*}
        \|f\natural g\|_{W(L^{p,q}_m(\mathbb{R}^{2d}))} \leq C \|f\|_{L^{p,q}_m(\mathbb{R}^{2d})} \|g\|_{W(L^1_v(\mathbb{R}^{2d}))}
    \end{align*}
    where $C$ is a constant depending on $m$ and $v$.
\end{lemma}
The correspondence principle of coorbit theory gives the characterisation of the $M^{p,q}_m(\mathbb{R}^d)$ spaces as those for which the map
\begin{align*}
    V_g: M^{p,q}_m(\mathbb{R}^{d}) \to \{F\in L^{p,q}_m(\mathbb{R}^{2d}): F \natural  V_g g = F\}.
\end{align*}
is an isomorphism for any $g\in M^1_v(\mathbb{R}^d)$. The twisted convolution $\natural$ for two functions $F,G$ is defined by
\begin{align*}
    F\natural G := \int_{\mathbb{R}^{2d}} F(z')G(z-z') e^{-2\pi i x'(\omega-\omega')}\, dz',
\end{align*}
and satisfies Young's inequalities for mixed-norm spaces;
\begin{align*}
     \|F\natural H \|_{L^{p,q}_m} \leq C_{m,v}\|F\|_{L^1_v} \|H\|_{L^{p,q}_m},
\end{align*}
for $F\in L^1_v(\mathbb{R}^{2d})$ and $H\in L^{p,q}_m(\mathbb{R}^{2d})$.

\subsection{Frames and Gabor Analysis}
Much of the appeal of time-frequency analysis, and coorbit theory more generally, lies in the discretisation results available. The STFT on a discrete subset of $\mathbb{R}^{2d}$ gives a basis like representation of a signal, known as a frame. Given a Hilbert space $\mathcal{H}$, a frame is a subset $\{f_i\}_{i\in I}$ satisfying the property known as the frame condition; 
\begin{align}
    A\|f\|^2_{\mathcal{H}} \leq \sum_{i\in I} \langle f, f_i\rangle^2_{\mathcal{H}} \leq B\|f\|^2_{\mathcal{H}}
\end{align}
for every $f\in\mathcal{H}$. An orthonormal basis satisfies the frame with frame bounds $A=B=1$, but more general frames need not give unique frame coefficients $\langle f, f_i\rangle$, and indeed this overcompleteness is in some senses desirable. The frame operator 
\begin{align*}
    Ef := \sum_{i\in I} \langle f, f_i\rangle_{\mathcal{H}}\, f_i
\end{align*}
is positive, bounded and invertible on $\mathcal{H}$ for a frame $\{f_i\}_{i\in I}$. The frame operator can be decomposed into the analysis operator $C:\mathcal{H}\to \ell^2$ defined by
\begin{align*}
    C f := \{\langle f, f_i\rangle_{\mathcal{H}} \}_{i\in I},
\end{align*}
and the synthesis operator $D: \ell^2 \to \mathcal{H}$ defined by
\begin{align*}
    D a := \sum_{i\in I} a_i f_i.
\end{align*}
Given a frame $\{f_i\}_{i\in I}$, the positivity and invertibility of $E$ give the existence of a dual frame $\{\Tilde{f}_i\}_{i\in I}$, given by $\Tilde{f}_i = E^{-1}f_i$, which by construction satisfies the identities
\begin{align*}
    f = \sum_{i\in I} \langle f, \Tilde{f}_i\rangle_{\mathcal{H}}\, f_i = \sum_{i\in I} \langle f, f_i\rangle_{\mathcal{H}}\, \Tilde{f}_i.
\end{align*}
In Gabor analysis, one constructs Gabor frames for $L^2(\mathbb{R}^d)$, of the form
\begin{align*}
    \{\pi(\lambda)g\}_{\lambda\in\Lambda}
\end{align*}
for some discrete set $\Lambda\in \mathbb{R}^{2d}$. Note that for the whole continuous space $\{\pi(z)g\}_{z \in\mathbb{R}^{2d}}$ is trivially a frame with $A=B$ and dual frame $\frac{g}{\|g\|^2_2}$. 
Moreover, the modulation spaces $M^{p,q}_m(\mathbb{R}^{d})$ are characterised by their frame coefficients. Namely given a Gabor frame for $L^2(\mathbb{R}^d)$ generated by window $g\in M^1_v(\mathbb{R}^d)$ , then for $f\in M^{\infty}$
\begin{align*}
    f \in M^{p,q}_m(\mathbb{R}^{d}) \iff \{\langle f, \pi(\lambda)g\rangle \}_{\lambda\in\Lambda} \in \ell^{p,q}_m.
\end{align*}

\subsection{Pseudo-Differential Operators}
We begin by defining the trace of an operator $S\in \mathcal{L}(L^2(\mathbb{R}^d);L^2(\mathbb{R}^d))$ as 
\begin{align*}
    \mathrm{tr}(S) := \sum_n \langle S e_n, e_n\rangle_{L^2},
\end{align*}
where $\{e_n\}_{n\in\mathbb{N}}$ is an orthonormal basis for $L^2(\mathbb{R}^d)$ and the trace does not depend on the choice of basis. The trace class of operators corresponds to those operators with finite trace of their positive part;
\begin{align*}
    \mathcal{S}^1:= \{S\in\mathcal{L}(L^2(\mathbb{R}^{d})): \mathrm{tr}(S) < \infty\},
\end{align*}
and is a Banach space and ideal of the bounded operators. The Hilbert-Schmidt operators are defined as the operators 
\begin{align*}
    \mathcal{HS}:= \{T\in\mathcal{L}(L^2(\mathbb{R}^{d})): T^*T\in\mathcal{S}^1 \}.
\end{align*}
The Hilbert-Schmidt operators form a Hilbert space equipped with the inner product 
\begin{align*}
    \langle S,T\rangle_{\mathcal{HS}} = \mathrm{tr}(ST^*)
\end{align*}
and contains $\mathcal{S}^1$ as a proper ideal. Since trace class and Hilbert-Schmidt operators (and more generally any Schatten class operator defined by $\ell^p$ decay of its singular values) are compact, they admit a spectral decomposition
\begin{align*}
    S = \sum_{n\in\mathbb{N}} \lambda_n \psi_n \otimes \phi_n,
\end{align*}
where $\lambda_n$ are the singular values of $S$, $\{\psi_n\}_{n\in \mathbb{N}}$ and $\{\phi_n\}_{n\in \mathbb{N}}$ are orthonormal sets and the sum converges in operator norm. The rank one operator here is defined as
\begin{align*}
    \psi \otimes \phi:= \langle \cdot,\phi\rangle_{L^2} \psi.
\end{align*}
For an operator $S\in\mathcal{L}(L^2(\mathbb{R}^d);\,L^2(\mathbb{R}^d))$, we can assign a function, or more generally a distribution, to the operator in several ways. The integral kernel of an operator $S$ is defined as the tempered distribution $K_S$ such that
\begin{align*}
    \langle S f, g \rangle_{\mathscr{S}^\prime,\mathscr{S}} = \langle K_S, g\otimes f \rangle_{\mathscr{S}^\prime,\mathscr{S}}.
\end{align*}
For a Hilbert-Schmidt operator $S$, the kernel $K_S \in L^2(\mathbb{R}^{2d})$, and conversely any kernel $K_S \in L^2(\mathbb{R}^{2d})$ defines a Hilbert-Schmidt operator \cite{pool66}. Alternatively, one can consider the Weyl symbol of an operator, $\sigma_S$. In the context of quantum time-frequency analysis, the Weyl symbol is best understood as the extension of the Wigner distribution in the rank one case to general operators. Recall that the Wigner distribution for functions $f,g\in L^2(\mathbb{R}^d)$ is given by
\begin{align*}
    W(f,g)(z) := \int_{\mathbb{R}^{d}} f\Big(x+\frac{t}{2}\Big)\overline{g\Big(x-\frac{t}{2}\Big)}e^{-2\pi i \omega t}\, dt.
\end{align*}
The Weyl symbol of the rank one operator $f\otimes g$ is then given by
\begin{align*}
    \sigma_{f\otimes g} = W(f,g),
\end{align*}
and extended linearly to the Hilbert-Schmidt operators. Due to Moyal's identity for the Wigner distribution, we find that
\begin{align*}
    \langle L_{\sigma}f, g \rangle_{L^2} = \langle \sigma, W(f,g) \rangle_{L^2}
\end{align*}
where $L_{\sigma}$ is the unique Hilbert-Schmidt operator with symbol $\sigma$. The Weyl quantisation $\sigma\to L_{\sigma}$ is unitary from $L^2(\mathbb{R}^{2d})$ to $\mathcal{HS}$, and by considering the Schwartz functions, extended by duality to tempered distributions. Finally, we can consider the spreading representation of an operator. Given a function $f\in L^2(\mathbb{R}^{2d})$, the operator given by
\begin{align*}
    H_f = \int_{\mathbb{R}^{2d}} f(z)\pi(z)\, dz,
\end{align*}
where the integral can be understood weakly, defines a Hilbert Schmidt operator, and conversely any Hilbert Schmidt operator $S$ has a unique spreading function $\eta_S$. In the rank one case, the spreading function is given by the ambiguity function:
\begin{align*}
    \eta_{f\otimes g}(z) = e^{-i\pi x\omega} V_g f(\omega,-x).
\end{align*}
Similarly to the kernel and Weyl quantisations, the spreading representation can be extended to the space of tempered distributions. We can consider how the various quantisation schemes interact, and find that for an operator $S\in\mathcal{L}(\mathscr{S};\mathscr{S}^\prime)$;
\begin{align*}
    \sigma_S &= \mathcal{F}_2 \mathcal{T} K_S
\end{align*}
and
\begin{align*}
    \eta_S &= \mathcal{F}_{\Omega} (\sigma_S).
\end{align*}
Here the symplectic Fourier transform
\begin{align*}
    \mathcal{F}_{\Omega} f(z) := \int_{\mathbb{R}^{2d}} f(z')e^{-2\pi i (x'\omega - x\omega ')}\, dz'
\end{align*}
is the appropriate notion of a Fourier transform on phase space. Importantly, the $M^p(\mathbb{R}^d)$ spaces for $1\leq p \leq \infty$ are invariant under the above transformations, and hence the spaces of operators formed by the $M^p(\mathbb{R}^d)$ spaces via the quantisation schemes discussed coincide.

\subsection{Quantum Harmonic Analysis}
Quantum time-frequency analysis as discussed in this paper can be seen as an extension of Quantum Harmonic Analysis (QHA), introduced by Werner in \cite{werner84}, and recently considered through the lens of time-frequency analysis in \cite{skrett18} \cite{skrett19} \cite{skrett20}, where tools from QHA are used to generalise known results and provide new insights by extending the mechanics of harmonic analysis to operators. A key insight in QHA is that the unitary representation of the Weyl-Heisenberg group on the Hilbert-Schmidt operators given by
\begin{align*}
    \alpha_z(S) := \pi(z)S\pi(z)^*,
\end{align*}
amounts to a translation of the Weyl symbol of the operator. That is to say,
\begin{align*}
    \sigma_{\alpha_z(s)} = T_z \sigma_S.
\end{align*}
If one then considers the trace as the operator analogue to the integral of a function, convolutions between operators and functions are naturally defined in the following manner:
\begin{definition}
For $f\in L^p(\mathbb{R}^{2d})$, $S\in\mathcal{S}^q$ and $T\in\mathcal{S}^p$, where $\frac{1}{p}+\frac{1}{q}=1+\frac{1}{r}$, convolutions are defined by
\begin{align*}
    f \star S &:= \int_{\mathbb{R}^{2d}} f(z) \alpha_z(S)\, dz \\
    S \star T(z) &:= \mathrm{tr}(S\alpha_z(\Check{T}))
\end{align*}
where $\Check{T}=PTP$ where $P$ is the parity operator. The first integral is to be interpreted as a Bochner integral. 
\end{definition}
Alongside convolutions, the central operation in QHA is the Fourier-Wigner transform $\mathcal{F}_W:\mathcal{HS}\to L^2(\mathbb{R}^{2d})$, defined as 
\begin{align*}
    \mathcal{F}_W (S)(z) := e^{-i\pi x \omega}\mathrm{tr}\big(\pi(-z)\beta_w (S)\big).
\end{align*}
The Fourier-Wigner of an operator gives the operator's spreading function. The convolutions and Fourier-Wigner transform interact as one would hope. Convolutions are associative and commutative, and satisfy the Fourier convolution property:
\begin{align*}
    \mathcal{F}_W(f\star S) &= \mathcal{F}_{\Omega}(f)\cdot\mathcal{F}_W (S) \\
    \mathcal{F}_{\Omega}(T \star S) &= \mathcal{F}_W (T) \cdot \mathcal{F}_W (S).
\end{align*}

\section{A Projective Representation on Hilbert-Schmidt Operators}\label{tfshifts}
In Werner's Quantum Harmonic Analysis framework, one constructs a unitary representation $\alpha_z$ on Hilbert-Schmidt operators. The effect of this representation on the Weyl symbol of an operator is a translation by $z$ in phase space, hence the representation is \textit{not} projective. Motivated by time-frequency analysis of functions/signals, we introduce a new representation on the space of Hilbert-Schmidt operators:
\begin{definition}\label{alphawzdef}
    Let $w=(w_1,w_2),z=(z_1,z_2)\in\mathbb{R}^{2d}$, and $S\in\mathcal{HS}$. The representation $\gamma_{w,z}$ is then defined as
\begin{align*}
    \gamma_{w,z} (S) := \pi(z)S\pi(w)^*.
\end{align*}
\end{definition}
The representation $\gamma_{w,z}$ is hence a \textit{projective} representation of $\mathbb{H}\times\mathbb{H}$ on $\mathcal{HS}\cong L^2(\mathbb{R}^d)\otimes L^2(\mathbb{R}^d)$, which we can see by a simple calculation: Using the notation of \cref{alphawzdef},
\begin{align*}
    \gamma_{w',z'}\big(\gamma_{w,z}(S)\big) = e^{-2\pi i (z_2\cdot z_1' - w_2'\cdot w_1)}\gamma_{w+w',z+z'}(S).
\end{align*}
To elucidate this representation, consider the rank one example $S=f\otimes g$. The action of $\gamma_{w,z}$ on $S$ is then 
\begin{align*}
    \gamma_{w,z} (S) = (\pi(z)f)\otimes (\pi(w)g).
\end{align*}

\subsection{Matrix coefficients and Integrated Representation}
Given our representation $\gamma$ on the Hilbert-Schmidt operators, it is natural to consider the matrix coefficient of $\gamma$. As the Hilbert-Schmidt operators form a Hilbert space, we can consider the matrix coefficients weakly, and denote the resulting form by $Q$, as the Polarised Cohen's class:
\begin{definition}
    Given $S,T\in\mathcal{HS}$, the \textit{Polarised Cohen's class} is given by the function $Q_S T: L^2(\mathbb{R}^{4d})\to \mathbb{C}$, defined as
    \begin{align*}
        Q_S T(w,z) := \langle T, \gamma_{w,z}(S)\rangle_{\mathcal{HS}}.
    \end{align*}
\end{definition}
The nomenclature polarised Cohen's class refers to the classification of a certain class of quadratic time-frequency representations, (partly) extended to operators in \cite{skrett19}. To see how the the polarised Cohen's class looks in practice, we consider the rank-one example:
\begin{example}
Let $S=f \otimes g$ and $T=\psi \otimes \phi$. Then 
\begin{align*}
    Q_{f \otimes g} (\psi \otimes \phi)(z,w) &= \langle \psi \otimes \phi, \gamma_{w,z}(f \otimes g)\rangle_{\mathcal{HS}} \\
    &= \mathrm{tr}((\psi \otimes \phi)\pi(w)(f \otimes g)^*\pi(z)^*) \\
    &= V_{f} \psi(z) \cdot \overline{V_{g} \phi(w)}.
\end{align*}
as we would expect for a polarised version of the rank-one Cohen's class (spectrogram).
\end{example}

Crucially, by polarising the general form of the Cohen's class, we move from a quadratic time-frequency representation to a linear one. In fact, the polarised Cohen's class is not only a linear form but also an isometry:
\begin{proposition}
    The mapping $T\mapsto Q_S (T)$ is an isometry from $\mathcal{HS}$ to $L^2(\mathbb{R}^{4d})$ for every $S\in \mathcal{HS}$ with $\|S\|_{\mathcal{HS}}=1$.
\end{proposition}

\begin{proof}
For $S,T\in \mathcal{HS}$, we decompose spectrally as $S=\sum_n \lambda_n f_n \otimes g_n$ and $T=\sum_n \eta_n \psi_n \otimes \phi_n$, where $\|f_i\|=\|g_i\|=\|\psi_i\|=\|\phi_i\|=1$. Then proceeding as in the previous example;
\begin{align}\label{specdecomp}
\begin{split}
    Q_S(T)(z,w) &= \langle T, \gamma_{w,z}(S)\rangle_{\mathcal{HS}} \\
    &= \sum_{n,m} \overline{\lambda_n}\eta_m V_{f_n} \psi_m (z) \cdot \overline{V_{g_n} \phi_m(w)},
\end{split}
\end{align}
from which it follows that 
\begin{align*}
    |Q_S(T)(z,w)|^2 = \sum_{i,j,n,m} \overline{\lambda_n}\lambda_i\eta_m\overline{\eta_i} V_{f_n}\psi_m(z)\overline{V_{f_i}\psi_j(z)} \cdot V_{g_n}\phi_m(w)\overline{V_{g_i}\phi_j(w)}.
\end{align*}
The proposition then follows from Moyal's orthogonality relation for functions;
\begin{align*}
    \int_{\mathbb{R}^{4d}} |Q_S(T)(z,w)|^2 \, dz \, dw &= \sum_{i,j,n,m} \overline{\lambda_n}\lambda_i\eta_m\overline{\eta_i} \int_{\mathbb{R}^{2d}} V_{f_n}\psi_m(z)\overline{V_{f_i}\psi_j(z)} \int_{\mathbb{R}^{2d}} V_{g_n}\phi_m(w)\overline{V_{g_i}\phi_j(w)}\, dw\, dz \\
    &=  \sum_{i,j,n,m} \overline{\lambda_n}\lambda_i\eta_m\overline{\eta_i} \delta_{n,i}\delta_{m,j} \\
    &= \|T\|^2_{\mathcal{HS}}\cdot\|S\|^2_{\mathcal{HS}}.
\end{align*}

\end{proof}
In fact following the approach of the above proof we have an even stronger statement, a Quantum Time-Frequency Analysis Moyal's identity: 
\begin{proposition}
    For $R,S,T,W\in\mathcal{HS}$;
\begin{align*}
    \langle Q_R (S), Q_T (W)\rangle_{L^2(\mathbb{R}^{4d})} = \langle R, T\rangle_{\mathcal{HS}}\cdot \overline{\langle S, W\rangle}_{\mathcal{HS}}.
\end{align*}
\end{proposition}

As a simple corollary we observe:
\begin{corollary}
    Given a non-zero $S\in\mathcal{HS}$, 
    \begin{align*}
        \overline{span}(\{\gamma_{w,z}(S)\}_{w,z\in\mathbb{R}^{2d}}) = \mathcal{HS}.
    \end{align*}
\end{corollary}
In quantum harmonic analysis, the density of translates of an operator depends on the zeros of its Fourier Wigner transform \cite{skrett18}, in analogue to the case of translates of functions. In the quantum time-frequency setting, density of time-frequency shifts of an operator requires only that the operator is non-zero. The density of translates of an operator also informs the injectivity of operator operator convolutions. In practice, this means for a sufficiently nice operator $S$, the map $T\mapsto Q_S T(z,z)$ is injective in $\mathcal{HS}$. However, no such operator exists for which the map $T\mapsto Q_S T(\lambda,\lambda)$ for some discrete lattice $\Lambda\subset \mathbb{R}^{2d}$ is injective in $\mathcal{HS}$, due to the failure of frames of translates for $L^2$. We will see this shortcoming is avoided by lifting the representation to double-phase space and consider a Schr\"odinger representation of the double-phase space on the Hilbert-Schmidt operators. 

Taking now a fixed $S\in\mathcal{HS}$, and considering the polarised Cohen's class as the mapping $Q_S:\mathcal{HS}\to L^2(\mathbb{R}^{4d})$, we can construct the adjoint. For some $F\in L^2(\mathbb{R}^{4d})$, this is can be done directly by 
\begin{align*}
    \langle T, Q_S^* F \rangle_{\mathcal{HS}} &= \langle Q_S T, F \rangle_{L^2} \\
    &= \int_{\mathbb{R}^{4d}} \langle T, \gamma_{w,z}(S)\rangle_{\mathcal{HS}}\overline{F(w,z)}\, dz\, dw \\
    &= \langle T, \int_{\mathbb{R}^{4d}} F(w,z) \gamma_{w,z}(S)\, dz\, dw\rangle_{\mathcal{HS}}
\end{align*}
where the vector-valued integral can be understood weakly. 

An important feature of the polarised Cohen's class is the reproducing property:
\begin{proposition}
    Let $S\in\mathcal{HS}$. Then For any $T\in\mathcal{HS}$, we have the identity
    \begin{align*}
        Q_S^* Q_S T = T.
    \end{align*}
    Furthermore, this gives rise to the reproducing formula
    \begin{align*}
        Q_S T(w,z) = \int_{\mathbb{R}^4d} Q_S T(w',z')k_{w,z}(w',z')\, dw'\, dz',
    \end{align*}
    where $k_{w,z}(w',z') = \langle \pi(z')S\pi(w')^*,\pi(z)S\pi(w)^*\rangle_{\mathcal{HS}}$.
\end{proposition}
\begin{proof}
    Writing $S,T$ as orthogonal decompositions $S=\sum_n f_n \otimes g_n$ and $T=\sum_n \psi_n \otimes \phi_n$, we calculate directly
    \begin{align*}
        Q_S^*Q_S T &= \int_{\mathbb{R}^{4d}} \sum_{n,m} V_{f_n} \psi_m(z)\cdot \overline{V_{g_n} \phi_m(w)}\cdot \gamma_{w,z}(S)\, dz\, dw \\
        &= \sum_{i,n,m} \left(\int_{\mathbb{R}^{2d}} V_{f_n}\psi_m(z) \pi(z)f_i\, dz\right)\otimes \left(\int_{\mathbb{R}^{2d}} V_{g_n}\phi_m(w) \pi(w)g_i\, dw\right)
    \end{align*}
    The identity follows clearly in the case of rank-one $S$. For higher rank $S$, we have the terms $V_{f_n}^* V_{f_m} \psi_i(w)$ (resp. g) for $n\neq m$, which are necessarily zero since $f_i$'s (resp $g_i$'s) are pairwise orthogonal as the spectral components of $S$. The claim then follows from Moyal's identity for the spaces $V_{f_i}(L^2)$, $V_{f_j}(L^2)$, along with the isometry property of $V_{f_i}$. Hence our reproducing kernel is given by 
    \begin{align*}
        k_{w,z}(w',z') = \langle \pi(z')S\pi(w')^*,\pi(z)S\pi(w)^*\rangle_{\mathcal{HS}},
    \end{align*}
    
\end{proof}
In the case of a rank-one $S$, this reproducing property was used in \cite{hern22}. We can formulate the reproducing property as a twisted convolution which may be defined in the following manner:
\begin{definition}
    Given $F,G\in L^2(\mathbb{R}^{4d})$, the twisted convolution $F\oast G$ is given by
    \begin{align*}
        F\oast G := \int_{\mathbb{R}^{4d}} F(w',z')G(w-w',z-z')e^{-2\pi i \big(z'_1(z_2-z'_2) - w'_1(w_2-w'_2) \big)}.
    \end{align*}
\end{definition}
A simple calculation then gives the identity:
\begin{align}\label{twistedrelation}
    Q_S T \oast Q_R W(w,z) = \langle S, W \rangle_{\mathcal{HS}} Q_R T(w,z),
\end{align}
or in the case $\|S\|_{\mathcal{HS}}=1$;
\begin{align*}
    Q_S T \oast Q_S S(w,z) = Q_S T(w,z).
\end{align*}
Using the same approach as for the $\mathbb{R}^{2d}$ case (cf. \cite{grochenigtfa}), the twisted convolution on $\mathbb{R}^{4d}$ satisfies a mixed Young's inequality. Namely, given functions $F\in L^1_v(\mathbb{R}^{4d})$ and $H\in L^{p,q}_m(\mathbb{R}^{4d})$ we have
\begin{align*}\label{twistedyoungs}
     \|F\oast H \|_{L^{p,q}_m} \leq C_{m,v}\|F\|_{L^1_v} \|H\|_{L^{p,q}_m},
\end{align*}
where $v$ is some sub-multiplicative function and $m$ a $v$-moderate weight, and $C_{m,v}$ a constant depending on $v$ and $m$.

Equipped with a reproducing kernel Hilbert space, it is natural to consider that Toeplitz operators on it. In our case we find a notion of localisation operator for operators:
\begin{definition}
    Given a function $F\in L^{\infty}(\mathbb{R}^{4d})$ and a window $S\in \mathcal{HS}$, the corresponding localisation operator $A_F^S:\mathcal{HS}\to \mathcal{HS}$ is given by
    \begin{align*}
        \int_{\mathbb{R}^{4d}} F(w,z) Q_S T(w,z)\gamma_{w,z}(S)\, dw\, dz.
    \end{align*}
\end{definition}
Using that the projection onto the RKHS is given by $Q_S Q_S^*:L^2(\mathbb{R}^{4d})\to Q_S (\mathcal{HS})$, the Toeplitz operator $T_F$ takes the form
\begin{align*}
    T_F (G) = Q_S Q_S^* (F\cdot G)
\end{align*}
for $G\in Q_S (\mathcal{HS})$. It can hence be shown that the localisation operator $A_F^S$ is unitarily equivalent to the Toeplitz operator $T_F$ on $Q_S(\mathcal{HS})$:
\begin{align*}
    A_F^S T &= \int_{\mathbb{R}^{4d}} F(w,z) Q_S T(w,z)\gamma_{w,z}(S)\, dw\, dz \\
    &= Q_S^*(F\cdot Q_S T) \\
    &= Q_S^* \big(Q_S Q_S^*( F\cdot Q_S T)\big) \\
    &= Q_S^* T_F Q_S T.
\end{align*}
As in the function case, the localisation of an operator may be preferable to for example a simple restriction and reconstruction on the Fourier Wigner of an operator. We can also consider the localisation operator weakly, giving
\begin{align*}
    \langle A_F^S T, R \rangle_{\mathcal{HS}} = \langle F\cdot Q_S T, Q_S R \rangle_{\mathcal{HS}},
\end{align*}
for $T,R\in\mathcal{HS}$.

\subsection{Time-Frequency Analysis of Operators}
We begin the study of time-frequency analysis for operators by revisiting the weak definition of Weyl quantisation. If we understand the Weyl quantisation of a symbol $\sigma\in L^2(\mathbb{R}^{2d})$ as the unique operator satisfying
\begin{align*}
    \langle L_{\sigma} f, g\rangle_{L^2(\mathbb{R}^{d})} = \langle \sigma, W(f,g)\rangle_{L^2(\mathbb{R}^{2d})},
\end{align*}
for all $f,g\in L^2(\mathbb{R}^{d})$, then we can equivalently define it by
\begin{align*}
    \langle L_{\sigma},  f \otimes g\rangle_{\mathcal{HS}} = \langle \sigma, W(f,g)\rangle_{L^2(\mathbb{R}^{2d})},
\end{align*}
for all $f,g\in L^2(\mathbb{R}^{d})$, or indeed by recognising that $W(f,g) = \sigma_{f\otimes g}$, by the condition
\begin{align*}
    \langle L_{\sigma},  S \rangle_{\mathcal{HS}} = \langle \sigma, \sigma_S \rangle_{L^2(\mathbb{R}^{2d})}
\end{align*}
for all $S\in\mathcal{HS}$. As discussed, translates on the Weyl correspond to $\alpha_z$ shifts on the operator, which has led to the study of quantum harmonic analysis. In order to examine the time-frequency analysis of operators, we must therefore understand the effect of modulation of the symbol, and moreover how the composition of the two affects the operator. To that end we define $\beta_w$:
\begin{definition}
    Let $w\in\mathbb{R}^{2d}$, and $S\in\mathcal{HS}$. Then
    \begin{align*}
        \beta_w(S) := e^{-\pi i w_1 w_2/2}\pi\Big(\frac{w}{2}\Big)S\pi\Big(\frac{w}{2}\Big).
    \end{align*}
\end{definition}
This is equivalent to the definition
\begin{align*}
    \beta_w(S) := \pi\Big(\frac{w}{2}\Big)S\pi\Big(-\frac{w}{2}\Big)^*.
\end{align*}
We can roughly intuit operator modulation as reflecting in phase space the action of the operator, the parity operator is namely invariant under operator modulation. The concept of an operator modulation is introduced implicitly in \cite{Gr76} and used in \cite{Ful23}. We claim that $\beta_w$ corresponds to a modulation of the Weyl symbol of an operator:
\begin{proposition}\label{symbolmodulation}
    Given $w\in\mathbb{R}^{2d}$, and $S\in\mathcal{HS}$,
    \begin{align*}
        \sigma_{\beta_w(S)} = M_w \sigma_S,
    \end{align*}
    where $M_w$ is the symplectic modulation; $M_w F(z) = e^{2\pi i \Omega(z,w)}F(z)$.
\end{proposition}
\begin{proof}
Recalling that 
\begin{align*}
    \sigma_{S} = \mathcal{F}_{\sigma}\mathcal{F}_W (S),
\end{align*}
it is clear that a modulation of the symbol must result in a translation of the Fourier-Wigner transform of the operator. To see that $\beta_w$ is the appropriate operator to do so, one can calculate the Fourier-Wigner transform of $\beta_w(S)$ directly:
\begin{align*}
    \mathcal{F}_W (\beta_w(S)) &= e^{-i\pi x \omega}\mathrm{tr}\big(\pi(-z)\beta_w (S)\big) \\
    &= e^{-\pi i w_1 w_2/2-i\pi x \omega}\mathrm{tr}\big(\pi\Big(\frac{w}{2}\Big)\pi(-z)\pi\Big(\frac{w}{2}\Big)S\big) \\
    &=  e^{-i\pi (z_1-w_1)(z_2-w_2)}\mathrm{tr}(\pi(-(z-w))S) \\
    &= \mathcal{F}_W (S)(z-w).
\end{align*}

\end{proof}
With a concept of modulation for operators, it follows that $\gamma_{w,z}$ shifts are an appropriate time-frequency shifts for operators, motivating our terminology of Quantum Time-Frequency Analysis: 
\begin{theorem}\label{tfshiftsymbol}
    Let $w=(w_1,w_2),z=(z_1,z_2)\in\mathbb{R}^{2d}$, and $S\in\mathcal{HS}$. Then
    \begin{align*}
        \sigma_{\gamma_{w,z}(S)} = e^{i\pi(w_1 + z_1)(w_2 - z_2)}\pi(U(w,z))\sigma_S,
    \end{align*}
    where
    \begin{align*}
        U(w,z) &= U(w_1,w_2,z_1,z_2) := \Big(\frac{w_1+z_1}{2},\frac{w_2+z_2}{2}, w_2 - z_2, z_1-w_1\Big).
    \end{align*}
\end{theorem}
\begin{proof}
    Rewriting $\gamma_{w,z}$, one finds 
    \begin{align*}
        \gamma_{w,z}(S) &= \pi(z)S\pi(w)^* \\
        &= e^{\pi i z_1 z_2 / 2}\pi\Big(\frac{z}{2}
        \Big)\pi\Big(\frac{z}{2}\Big) S \pi(w)^*\pi\Big(\frac{z}{2}\Big)^*\pi\Big(\frac{z}{2}\Big) \\
        &= e^{\pi i z_1 z_2} \beta_z \Big( \pi\Big(\frac{z}{2}\Big) S \pi(w)^*\pi\Big(\frac{z}{2}\Big)^* \Big) \\
        &= e^{\pi i z_1 z_2-\pi i w_1 w_2 / 2}\beta_z \Big( \pi\Big(\frac{z}{2}\Big) \pi\Big(\frac{w}{2}\Big)^* \pi\Big(\frac{w}{2}\Big) S \pi\Big(\frac{w}{2}\Big)^* \pi\Big(\frac{w}{2}\Big)^*\pi\Big(\frac{z}{2}\Big)^* \Big) \\
        &= e^{\pi i \big(z_1 z_2 - w_1 w_2 / 2 - z_1 w_2 + w_1 z_2  \big)}\beta_z \Big(\pi\Big(\frac{w}{2}\Big)^* \pi\Big(\frac{z}{2}\Big)  \pi\Big(\frac{w}{2}\Big) S \pi\Big(\frac{w}{2}\Big)^* \pi\Big(\frac{z}{2}\Big)^* \pi\Big(\frac{w}{2}\Big)^*\Big) \\
        &= e^{\pi i \big(z_1 z_2 - w_1 w_2 - z_1 w_2 + w_1 z_2 \big)}\beta_{z-w} \Big(\pi\Big(\frac{w+z}{2}\Big) S \pi\Big(\frac{w+z}{2}\Big)^*\Big) \\
        &= e^{i\pi(z_1 + w_1)(z_2 - w_2)} \beta_{z-w}\alpha_{\frac{w+z}{2}} (S)
    \end{align*}
    
\end{proof}
In the rank one case, this recovers the well-known intertwining of the (cross) Wigner transform \cite{follandha}. It is also clear that in the two cases $z=w$ and $z=-w$, one recovers the pure translation and modulation operators respectively, as one would expect examining the formulas for $\gamma$, $\alpha$ and $\beta$. The converse then follows:
\begin{corollary}
    Let $w,z,S$ be as in \cref{tfshiftsymbol}. Then
    \begin{align*}
        L_{M_w T_z\sigma_S} = e^{-2\pi i w_1 z_1} \gamma_{U^{-1}(w,z)}(S)
    \end{align*}
    and
    \begin{align*}
        L_{T_z M_w\sigma_S} = e^{-2\pi i w_2 z_2} \gamma_{U^{-1}(w,z)}(S),
    \end{align*}
    were for later reference we remark that
    \begin{align*}
        U^{-1}(w,z)&=\Big(w_1-\frac{z_2}{2},w_2+\frac{z_1}{2},w_1+\frac{z_2}{2},z_2-\frac{w_1}{2}\Big) \\
        &= \Big(w-\frac{Jz}{2},w+\frac{Jz}{2}\Big).
    \end{align*}
\end{corollary}

From this relation between $\gamma$ shifts of operators and $\pi$ shifts of the corresponding Weyl symbol, we can consider the rank-one operator $f\otimes g$ with Weyl symbol $W(f,g)$, and find as a corollary the formula used in \cite{cordero03}:
\begin{corollary}[Lemma 2.2, \cite{cordero03}]
    Given $f,g,\psi,\phi\in L^2(\mathbb{R}^d)$;
    \begin{align*}
        V_{W(\psi,\phi)} \big(W(f,g)\big)(w,z) = e^{-2\pi i w_1 z_1} V_{\psi} f(w-\frac{Jz}{2})\cdot \overline{V_{\phi} g(w+\frac{Jz}{2})}.
    \end{align*}
\end{corollary}

The change of variables $U^{-1}$ is the analogue of the symmetric change of variables
\begin{align*}
    \mathcal{T}_S F(x,t) = F\Big(x+\frac{t}{2},x-\frac{t}{2}\Big),
\end{align*}
lifted to double (symplectic) phase space. This connection is unsurprising, since $\mathcal{T}_s$, followed by the partial Fourier transform $\mathcal{F}_2$, transforms the kernel of an operator to its Weyl symbol. It is interesting to note that $U$ is \textit{not} a symplectic matrix, as it takes the form
\begin{align*}
    U = \begin{pmatrix}
    0 & -1 & 0 & 1 \\
    1 & 0 & -1 & 0 \\
    \frac{1}{2} & 0 & \frac{1}{2} & 0\\
    0 & \frac{1}{2} & 0 & \frac{1}{2}
    \end{pmatrix}.
\end{align*}
$U$ is transformed to a symmplectic matrix (corresponding to the transformation of the integral kernel to th Weyl symbol) by composition with the permutation matrix 
\begin{align*}
    c_2 = \begin{pmatrix}
    0 & 0 & 1 & 0\\
    1 & 0 & 0 & 0\\
    0 & 0 & 0 & 1 \\
    0 & 1 & 0 & 0
    \end{pmatrix}.
\end{align*}
The permutation $c_2$ is the permutation considered in Theorem 3.6 of \cite{cordero19}, which we will later see is natural in the Quantum TFA setting as well, as it arises from the identity $K_{\pi(w)S\pi(z)^*} = \pi(c_2(w,z))K_S$. Since Weyl quantisation is unitary from $L^2(\mathbb{R}^{2d})$ to $\mathcal{HS}$, we can reformulate the definition of the polarised Cohen's class in terms of the respective Weyl symbols of the operators:
\begin{corollary}\label{cohensstftform}
    For $S,T\in\mathcal{HS}$ and $w,z\in\mathbb{R}^{2d}$,
    \begin{align}
        Q_S T(w,z) = e^{-i\pi(z_1+w_1)(z_2-w_2)} \langle \sigma_T, \pi(U(w,z))\sigma_S\rangle_{L^2}.
    \end{align}
\end{corollary}
The interpretation of $\gamma$ as a time-frequency shift for operators motivates the study of what we call quantum time-frequency analysis, as we uncover a wide range of tools from classical time-frequency analysis at our disposal. In the rest of this work we examine how such tools behave on the operator level, and connect many to examples of well-known results for functions.

The polarised Cohen's class also has an interesting connection to the operator STFT introduced in \cite{dorf22}:
\begin{proposition}
    Let $S,T\in\mathcal{HS}$. Then with $\mathfrak{V}_S$ as defined in \cite{dorf22};
    \begin{align*}
        Q_S T(w,z) = e^{i\pi w_1\cdot w_2}\mathcal{F}_W (\mathfrak{V}_S T(z))(-w).
    \end{align*}
\end{proposition}
\begin{proof}
    This follows from a direct calculation;
    \begin{align*}
        \mathcal{F}_W (\mathfrak{V}_S T(z))(-w) &= e^{-i\pi w_1\cdot w_2}\mathrm{tr}(\pi(w)\mathfrak{V}_S T(z)) \\
        &= e^{-i\pi w_1\cdot w_2}\mathrm{tr}(\pi(w)S^*\pi(z)^* T) \\
        &= e^{-i\pi w_1\cdot w_2}\langle T, \pi(z)S\pi(w)^* \rangle_{\mathcal{HS}}.
    \end{align*}
    
\end{proof}
As the Fourier-Wigner identifies an operator with its spreading function, we can interpret the polarised Cohen's class to be the pointwise spreading representation of the operator STFT.

\subsection{A Fourier Transform on Double Phase Space}
We define a Fourier Transform on double-phase space in the following manner:
\begin{definition}
    Given a function $F\in L^1(\mathbb{R}^{4d})$, we define the double-symplectic Fourier transform $\mathcal{F}_{\Phi}$ as
    \begin{align*}
        \mathcal{F}_{\Phi} (F)(w,z) := \int_{\mathbb{R}^{4d}} F(w',z')e^{-2\pi i (\Omega(z,z') - \Omega(w,w'))}\, dw\, dz.
    \end{align*}
\end{definition}
This definition can be extended to a unitary transformation on $L^2(\mathbb{R}^{4d})$ following the standard density argument, and is its own inverse: $\mathcal{F}_{\Phi}^2 = I$.

The double-symplectic Fourier transform can be seen as the natural transform on double phase space, we have for example the following relation for the polarised Cohen's class:
\begin{proposition}
    Given $S,T,R,W\in\mathcal{HS}$, the following relation holds:
    \begin{align*}
        \mathcal{F}_{\Phi} (Q_S T\cdot \overline{Q_R W})(w,z) = Q_S R(w,z)\cdot \overline{Q_T W}(w,z).
    \end{align*}
\end{proposition}
\begin{proof}
    A direct calculation shows
    \begin{align*}
        \mathcal{F}_{\Phi} (Q_S T\cdot \overline{Q_R W})(w,z) &= \int_{\mathbb{R}^{4d}} \langle T, \pi(z') S\pi(w')^*\rangle_{\mathcal{HS}}\cdot \overline{\langle R, \pi(z') W\pi(w')^*\rangle_{\mathcal{HS}}} e^{-2\pi i (\Omega(z,z') - \Omega(w,w'))}\, dw'\, dz' \\
        &= \int_{\mathbb{R}^{4d}} \langle T, \pi(z') S\pi(w')^*\rangle_{\mathcal{HS}}\cdot \overline{\langle R, \pi(z)^*\pi(z')\pi(z) W\pi(w)^*\pi(w')^*\pi(w)\rangle_{\mathcal{HS}}} \, dw'\, dz' \\
        &= \langle T, \pi(z)R\pi(w)^*\rangle_{\mathcal{HS}} \cdot \overline{\langle S, \pi(z)W\pi(w)^*\rangle_{\mathcal{HS}}},
    \end{align*}
    where we used Moyal's identity for the polarised Cohen's class, and the intertwining property for $\pi(z)$ shifts.
    
\end{proof}
From this identity we have the natural corollary
\begin{corollary}
    Given $S,T\in\mathcal{HS}$;
    \begin{align*}
        \mathcal{F}_{\Phi} (|Q_S T|^2)(w,z) = Q_S S(w,z) \cdot Q_T T(w,z).
    \end{align*}
\end{corollary}
We are also led to the identity:
\begin{proposition}
    Given $S,T\in\mathcal{HS}$;
    \begin{align*}
        \mathcal{F}_{\Phi} (Q_S T) = K_S(z_1,w_1)\overline{\widehat{K_T}(z_2,w_2)} e^{-2\pi i (z_1z_2-w_1w_2)}.
    \end{align*}
\end{proposition}
\begin{proof}
    This again follows from the corresponding identity on functions;
    \begin{align*}
        \mathcal{F}_{\Omega} \big(V_g f\big)(z) = f(z_1)\overline{\Hat{g}(z_2)}e^{-2\pi i z_1 z_2}.
    \end{align*}
    Decomposing the operators then as in the previous proposition, we find
    \begin{align*}
        \mathcal{F}_{\Phi} (Q_S T)(w,z) &= \sum_{m,n} \int_{\mathbb{R}^{2d}} V_{s^1_n} t^1_m(z') e^{-2\pi i \Omega(z,z')}\, dw \cdot \int_{\mathbb{R}^{2d}} \overline{V_{s^2_n} t^2_m(w')}e^{2\pi i \Omega(w,w')}\, dz \\
        &= \sum_{m,n} \mathcal{F}_{\Omega} (V_{s^1_n} t^1_m)(z)\cdot \mathcal{F}_{\Omega} (\overline{V_{s^2_n} t^2_m})(-w) \\
        &= \sum_{m,n} s^1_n(z_1)\cdot\overline{\widehat{t^1_m}}(z_2)\cdot \overline{s^2_n(w_1)}\cdot\widehat{t^2_m}(w_2)e^{-2\pi i (z_1z_2-w_1w_2)} \\
        &=  K_S(z_1,w_1)\overline{\widehat{K_T}(z_2,w_2)} e^{-2\pi i (z_1z_2-w_1w_2)}.
    \end{align*}
    
\end{proof}

\section{Extending to $\mathcal{M}^{p,q}_m$ Spaces}
As in the classical setting, it is useful to have a larger range of spaces than the Hilbert-Schmidt operators at one's disposal to take advantage of tools such as atomic decomposition. As in the function case this is done by using the dual pairing of the Schwartz space and tempered distributions, we will see that the correct spaces in the quantum setting are the Schwartz operators and their dual. We thus consider the dual pairing $(\mathfrak{S},\mathfrak{S}')=(\mathscr{S}\Hat{\otimes}_{\pi}\mathscr{S},\mathcal{L}(\mathscr{S},\mathscr{S}'))$. The duality of the projective tensor product of topological vector spaces is defined in terms of the generalised trace map, $\mathrm{Tr}:X \Hat{\otimes}_{\pi} X'\to \mathbb{C}$ by $\mathrm{Tr}:\sum_i x_i \otimes y_i \mapsto \sum_i y_i(x_i)$. 
 The space $X \Hat{\otimes}_{\pi} Y$ then has the dual $\mathcal{L}(X,Y')$, with the action:
\begin{align*}
    \langle u, z\rangle &= \mathrm{Tr}((u\otimes I_Y)(z)) \\
    &= \sum_n \lambda_i \langle u(x_i),y_i\rangle_{Y',Y}
\end{align*}
for $z=\sum_i \lambda_i x_i \otimes y_i \in X \Hat{\otimes}_{\pi} X'$. Hence for some $S=\sum_n f_n \otimes g_n\in \mathfrak{S}$, $T\in\mathfrak{S}'$, we define
\begin{align*}
    Q_S T &:= \langle T, \gamma_{w,z}(S)\rangle_{\mathfrak{S}',\mathfrak{S}} \\
    &= \sum_n \langle T\pi(w)f_n,\pi(z)g_n\rangle_{\mathscr{S}',\mathscr{S}}.
\end{align*}
In order for classical results to extend to the quantum setting, we often consider the behaviour of operators on the symbol level, and it is therefore useful to note that the Weyl quantisation of Schwartz functions and tempered distributions observes the same duality pairing, that is to say for $S\in\mathfrak{S}$, $T\in\mathfrak{S}'$,
\begin{align*}
    \langle S, T\rangle_{\mathfrak{S},\mathfrak{S}'} = \langle \sigma_S, \sigma_T \rangle_{\mathscr{S}',\mathscr{S}}.
\end{align*}
\begin{claim}
    Given $S\in\mathfrak{S}$, $T\in\mathfrak{S}'$ and $w,z\in\mathbb{R}^{2d}$, 
    \begin{align}
        Q_S T(w,z) = e^{-i\pi(z_1-w_1)(z_2+w_2)} \langle \sigma_T, \pi(U(w,z))\sigma_S\rangle_{\mathscr{S}',\mathscr{S}}.
    \end{align}
\end{claim}
\begin{proof}
    This follows from \cref{tfshiftsymbol}, along with the fact that Weyl quantisation is an isomorphism between $\mathscr{S}(\mathbb{R}^{2d})$ and $\mathfrak{S}$, and between $\mathscr{S}'(\mathbb{R}^{2d})$ and $\mathfrak{S}'$.
    
\end{proof}

With the definition of the polarised Cohen's class extended to Schwartz operators and their dual, we  define the general modulation spaces $\mathcal{M}^{p,q}_m$ for a $v$-moderate weight $m$;
\begin{definition}\label{mixedpqdef}
 Let $S_0 = \varphi_0 \otimes \varphi_0$. Then given $p=(p_1,p_2)$, $q=(q_1,q_2)$, $1\leq p_1,p_2,q_1,q_2 \leq \infty$;
    \begin{align*}
        \mathcal{M}^{p,q}_m := \{T\in \mathfrak{S}': Q_{S_0} T \in L^{p,q}_m(\mathbb{R}^{4d})\}.
    \end{align*}
    for any polynomial-growth weight function $m$, with norm $\|T\|_{\mathcal{M}^{p,q}_m} = \|Q_{S_0} T\|_{L^{p,q}_m}$.
\end{definition}
Using this definition, we can consider the $\mathcal{M}^{p,q}$ spaces as vector--valued Lebesgue-Bochner spaces in the following manner:
\begin{align*}
    \|T\|_{\mathcal{M}^{p,q}} &= \Big( \int_{\mathbb{R}^{2d}} \Big|\int_{\mathbb{R}^{2d}} |\langle T, \gamma_{w,z}(\varphi_0\otimes\varphi_0) \rangle_{\mathfrak{S}^\prime,\mathfrak{S}}|^{p}\, dw\Big|^{q/p}\, dz \Big)^{1/q} \\
    &= \Big( \int_{\mathbb{R}^{2d}} \Big|\int_{\mathbb{R}^{2d}} |\langle T\pi(w)\varphi_0,\pi(z)\varphi_0 \rangle_{\mathscr{S}^\prime,\mathscr{S}}|^{p}\, dw\Big|^{q/p}\, dz\Big)^{1/q} \\
    &= \Big( \int_{\mathbb{R}^{2d}} \Big|\int_{\mathbb{R}^{2d}} |V_{\varphi_0} \big(T^*\pi(z)\varphi_0\big)(w) |^{p}\, dw\Big|^{q/p}\, dz\Big)^{1/q} \\
    &= \Big( \int_{\mathbb{R}^{2d}} \| T^*\pi(z)\varphi_0 \|_{M^p}^{q}\, dz\Big)^{1/q} \\
    &= \|T^*\pi(z)\varphi_0\|_{L^q(\mathbb{R}^{2d};M^p(\mathbb{R}^d))}.
\end{align*}
Hence by identifying $T$ with its integral kernel, which itself can be seen as a vector--valued function, the spaces $\mathcal{M}^{p,q}$ are $M^p(\mathbb{R}^d)$--valued modulation spaces in the sense of \cite{Wa07}. We will need the following lemma in the sequel, which is a result of \cite{Wa07} applied to our particular vector--valued setting:
\begin{lemma}\label{wieneramalg}
    By identifying an operator $T\in\mathcal{M}^{p,q}$ with its integral kernel, we have
    \begin{align*}
        \mathcal{M}^{p,q} \cong \mathcal{F} W\big(\mathcal{F}L^q(\mathbb{R}^d;M^p(\mathbb{R}^d)),L^q(\mathbb{R}^d)\big).
    \end{align*}
\end{lemma}
We can equivalently define the $\mathcal{M}^{p,q}_m$ spaces as follows:
\begin{lemma}
    Given $p=(p_1,p_2)$, $q=(q_1,q_2)$, $1\leq p_1,p_2,q_1,q_2 \leq \infty$;
    \begin{align*}
        \mathcal{M}^{p,q}_m \cong M^{\Tilde{p},\Tilde{q}}_{\Tilde{m}}(\mathbb{R}^{2d}),
    \end{align*}
    where $M^{\Tilde{p},\Tilde{q}}_{\Tilde{m}}(\mathbb{R}^{2d}) = \{f\in M^{\infty}(\mathbb{R}^{2d}): V_{\varphi_0}f(U({\cdot,\cdot}))\in L^{p,q}_m\}$.
\end{lemma}
\begin{proof}
    The isomorphism is given by $L_{\sigma}\mapsto \sigma$, which along with \cref{cohensstftform} completes the proof.
    
\end{proof}
Just as the Hilbert-Schmidt operators are those operators $T$ for which $Q_{S_0} T \oast Q_{S_0} S_0 = Q_{S_0} T$, we can alternatively characterise the $\mathcal{M}^{p,q}_m$ spaces an analogous manner:
\begin{proposition}
    Given $T\in\mathfrak{S}^\prime$;
    \begin{align*}
        T\in\mathcal{M}^{p,q}_m \iff Q_{S_0} T \oast Q_{S_0} S_0 = Q_{S_0} T.
    \end{align*}
\end{proposition}
\begin{proof}
    This follows from recognising that for $F\in \mathscr{S}^\prime$, $F \oast Q_{S_0} S_0 = Q_S Q_S^* F$, and then proceeding in the same manner as the function case (cf. Chapter 11 \cite{grochenigtfa}) to show that the map $Q_S^*$ is bounded from $L^{p,q}_m(\mathbb{R}^{4d})$ to $\mathcal{M}^{p,q}_m$.
    
\end{proof}

The connection between the spaces $\mathcal{M}^{p,q}_m$ and the space of operators with Weyl symbol in $M^{p,q}_m(\mathbb{R}^{2d})$ is an interesting one, and one which indicates how the operators in $\mathcal{M}^{p,q}_m$ will behave. Indeed we have the relation \cref{tfshiftsymbol} between $\gamma$ shifts and time-frequency shifts of the Weyl symbol, and subsequently the integral kernel. However, since the change of variables $U$ does not correspond to a symplectic matrix, we cannot use the metaplectic intertwining property to consider the STFT with respect to some transformed window, or use the tools developed in \cite{Gia24}. Our definition is related to the \textit{Symplectic Modulation Spaces} introduced in \cite{goetz13}. Hence the subtle difference in order of integration in the $L^{p,q}$ norm between the $\mathcal{M}^{p,q}$ spaces and the STFT of kernels is an important one. The framework of the $\mathcal{M}^{p,q}$ spaces do seem to be important to understanding the mapping properties of operators between modulation spaces, and in \cite{cordero19} the authors have used the rank one version of the condition \cref{mixedpqdef} to characterise operators between $M^p(\mathbb{R}^d)$ and $M^{\infty}(\mathbb{R}^d)$, as well as between $M^1(\mathbb{R}^d)$ and $M^p(\mathbb{R}^d)$. 

In the case $p=q$, with appropriate $m$, the $L^p_m(\mathbb{R}^{2d})$ condition in \cref{mixedpqdef} imposes the same $M^p_m(\mathbb{R}^{2d})$ condition on the Weyl symbol and kernel. Critically, this means that $\mathcal{M}^1$ and $\mathcal{M}^{\infty}$ spaces are precisely the operators with Weyl symbols or kernels in $M^1(\mathbb{R}^{2d})$ and $M^{\infty}(\mathbb{R}^{2d})$ respectively, and we can use atomic decomposition for these spaces based on frames for the symbols. These operators correspond to the endpoints of the aforementioned Gelfand triple $(\mathcal{M}^1,\mathcal{HS},\mathcal{M}^{\infty})$. This Gelfand triple was investigated in \cite{feich98}, and has been an object of interest in many works since. Since the space $\mathcal{M}^1$ has Weyl symbol in $M^1(\mathbb{R}^{2d})$, many of the desirable features of $M^1(\mathbb{R}^{2d})$ will find their parallel in $\mathcal{M}^1$, one example of which we will see in the next section when considering frames. The space $\mathcal{M}^1$, which we will hereafter refer to as \textit{Feichtinger operators}, can be defined in several equivalent ways \cite{feich22}, which we recall here for convenience:
\begin{lemma}\label{feichopequiv}
    For any $T\in\mathcal{HS}$, the following are equivalent:
    \begin{enumerate}
        \item $T \in \mathcal{M}^1$
        \item $\sigma_T \in M^1(\mathbb{R}^{2d})$
        \item $Q_T T \in L^1(\mathbb{R}^{4d})$
        \item $Q_S T \in L^1(\mathbb{R}^{4d})$ for any non-zero $S\in \mathcal{M}^1$
        \item $T\in \mathcal{N}(M^1;M^1)$
    \end{enumerate}
\end{lemma}
It follows from the last condition that the space of Feichtinger operators is a proper subset of trace class operators.
On the other hand, the characterisation of $\mathcal{M}^{\infty}$ is also well known (cf. \cite{balasz2019}):
\begin{align}
    \mathcal{M}^{\infty} \cong \mathcal{L}(M^1(\mathbb{R}^d),M^{\infty}(\mathbb{R}^d)).
\end{align}
While the the operators of the Gelfand triple $(\mathcal{M}^1,\mathcal{HS},\mathcal{M}^{\infty})$ can be well described, the general $\mathcal{M}^{p,q}_m$ spaces may appear more opaque. We aim to elucidate how they behave, and to that end we consider frames for operators in the next section. 

Since we can interpret the mixed norm spaces $L^{p,q}_m$ as a vector-valued $L^q$ space, operators in $\mathcal{M}^{p,q}_m$ will be seen to act as operators mapping to $M^q(\mathbb{R}^{d})$ spaces. To give an intuition of this, consider the simple rank-one example:
\begin{example}
    Given $f\in M^{q}(\mathbb{R}^d), g\in M^{p}(\mathbb{R}^d)$, $f\otimes g \in \mathcal{M}^{p,q}$.
\end{example}
Using \cref{twistedyoungs}, we can show in the same manner to the function case that $S\in\mathcal{M}^1_v$ defines an equivalent norm on $\mathcal{M}^{p,q}_m$:
\begin{proposition}\label{equivnorms}
    Given $S\in\mathcal{M}^1_v$, $S$ defines an equivalent norm for $\mathcal{M}^{p,q}_m$ by
    \begin{align*}
        \|Q_S T\|_{L^{p,q}_m} \asymp \|T\|_{\mathcal{M}^{p,q}}.
    \end{align*}
\end{proposition}
\begin{proof}
    This follows from \cref{twistedrelation} and \cref{twistedyoungs};
    \begin{align*}
        \|Q_S T\|_{L^{p,q}_m} &= \|Q_{S_0} T \oast Q_S S_0\|_{L^{p,q}_m} \\
        &\leq C_{v,m}\|Q_{S_0} T\|_{L^{p,q}_m} \|Q_S S_0\|_{L^1_v}
    \end{align*}
    and conversely 
    \begin{align*}
        \|Q_{S_0} T\|_{L^{p,q}_m} &= \frac{1}{\|S\|_{\mathcal{HS}}^2} \|Q_{S} T \oast Q_{S_0} S\|_{L^{p,q}_m} \\
        &\leq \frac{C_{v,m}}{\|S\|_{\mathcal{HS}}^2}\|Q_{S} T\|_{L^{p,q}_m} \|Q_S S_0\|_{L^1_v}.
    \end{align*}
    
\end{proof}

\section{Discretisation of $\mathcal{M}^{p,q}$ spaces}
Since one of the pillars of time-frequency research is the result that one can discretise functions in modulation spaces by time-frequency shifts of an atom on a discrete subset of phase space, it is natural and desirable to pursue such results for quantum time-frequency analysis, and indeed is a core motivation. Since we can view $Q_S T$ as a wavelet transform on $\mathcal{HS}$, coorbit theory instructs us that there exist discrete decompositions for $T\in\mathcal{M}^{p,q}_m$, and we will spend this chapter considering these.
\subsection{Frames for $\mathcal{HS}$}
For the Hilbert space of Hilbert-Schmidt operators, we define the Gabor frames in the following manner:
\begin{definition}
    Let $S\in \mathcal{HS}$. We call $S$ a Gabor frame for $\mathcal{HS}$ on $\Lambda\times M$ if the set
    \begin{align*}
        \{ \gamma_{\lambda,\mu}(S) \}_{(\lambda,\mu\in \Lambda \times M)}
    \end{align*}
    is a frame for $\mathcal{HS}$.
\end{definition}
We will see that this definition is the natural concept for a Gabor frame for operators. We use the same term, Gabor frame, for both the operator and function case, but it will always be clear from the context which we refer to. For completeness we recall a central result of \cite{balasz2008}:
\begin{proposition}[Theorem 4.1(ii), \cite{balasz2008}]\label{balazsframe}
    Let $\{f_i\}_{i\in I}$ and $\{g_j\}_{j\in J}$ be frames for $L^2(\mathbb{R}^d)$. Then $\{f_i \otimes g_j\}_{(i,j)\in I\times J}$ is a frame for $\mathcal{HS}$.
\end{proposition}
The result follows from a simple calculation using the definitions of frames. Since the quantisation schemes are unitary, we immediately find that:
\begin{proposition}\label{functiontoopframe}
    Let $\{g_i\}_{i\in I}$ be a frame in $L^2(\mathbb{R}^{2d})$. Then $\{L_{g_i}\}_{i\in I}$ is a frame for $\mathcal{HS}$. 
\end{proposition}
\begin{proof}
    This is just a result of Weyl quantisation being unitary from $L^2(\mathbb{R}^{2d})$ to $\mathcal{HS}$, that is $\|S\|_{\mathcal{HS}} = \|\sigma_S\|_{L^2}$, and
    \begin{align*}
        \langle S, L_{g_i} \rangle_{\mathcal{HS}} = \langle \sigma_S, g_i \rangle_{L^2}.
    \end{align*}
    By our assumption on $\{g_i\}_{i\in I}$, we have that for any $S\in\mathcal{HS}$;
    \begin{align*}
        A\|\sigma_S\|_{L^2}^2 \leq \sum_{i\in I} |\langle \sigma_S, g_i \rangle_{L^2}|^2 \leq B\|\sigma_S\|_{L^2}^2.
    \end{align*}
    It then follows that
    \begin{align*}
        A\|S\|_{\mathcal{HS}}^2 \leq \sum_{i\in I} |\langle S, L_{g_i} \rangle_{\mathcal{HS}}|^2 \leq B\|S\|_{\mathcal{HS}}^2,
    \end{align*}
    precisely the frame condition for $\mathcal{HS}$.
    
\end{proof}
In the case of the frame $\{g_i\}_{i\in I}$ being a Gabor frame, that is, $g_i = \pi(w_i,z_i)g$ for some $g\in L^2(\mathbb{R}^{2d})$, we can find the particular form of the Gabor frame generated by $L_g$ using \cref{tfshiftsymbol}:
\begin{corollary}
    Given a Gabor frame $\{\pi(\lambda,\mu)g\}_{\lambda,\mu\in\Lambda\times M}$ for $L^2(\mathbb{R}^{2d})$, the Gabor system given by
    \begin{align*}
        \{\gamma_{\lambda,\mu}(L_{g})\}_{(\lambda, \mu)\in U^{-1}(\Lambda\times M)},
    \end{align*}
    is a frame in $\mathcal{HS}$.
\end{corollary}
Conversely, using the result of \cref{balazsframe} we can consider the Weyl symbol of a Hilbert-Schmidt frame constructed from two function frames:
\begin{proposition}\label{framesymbol}
    Given two Gabor frames $\{\pi(\lambda)g\}_{\lambda\in\Lambda}$, $\{\pi(\mu)f\}_{\mu\in M}$ in $L^2(\mathbb{R}^d)$, the Weyl symbol of the tensor product, $\sigma_{f\otimes g}$, generates a Gabor frame in $L^2(\mathbb{R}^{2d})$ on the lattice $U(\Lambda\times M)$.
\end{proposition}
\begin{proof}
    By \cref{balazsframe}, the tensor product $f\otimes g$ forms a Gabor frame for $\mathcal{HS}$ on the lattice $\Lambda\times M$. The result then follows in the converse manner to \cref{functiontoopframe} by again using \cref{tfshiftsymbol}. In particular, given any $h \in L^2(\mathbb{R}^{2d})$:
    \begin{align*}
        A\|L_h\|_{\mathcal{HS}}^2 \leq \sum_{\Lambda\times M} |\langle L_h, \pi(\mu)f\otimes\pi(\lambda)g \rangle_{\mathcal{HS}}|^2 \leq B\|L_h\|_{\mathcal{HS}}^2,
    \end{align*}
    and so
    \begin{align*}
        A\|h\|_{L^2}^2 \leq \sum_{\Lambda\times M} |\langle h, \pi(U(\lambda,\mu))\sigma_{f\otimes g} \rangle_{L^2}|^2 \leq B\|h\|_{L^2}^2.
    \end{align*}

\end{proof}

Finally we consider the dual frame of an operator. In \cite{balasz2008} the author showed that a dual frame for a rank-one tensor frame $\{f_i\otimes g_j\}_{I\times J}$ is given by the tensor product of dual frames $\{\Tilde{f}_i\otimes\Tilde{g_j}\}_{I\times J}$. Clearly in this case the dual window of the rank one $\mathcal{M}^1$ window is again in $\mathcal{M}^1$. Indeed, since operator frames correspond to function frames, we can use the Weyl correspondence to show:
\begin{proposition}
    Let $S\in\mathcal{M}^1$ generate a Gabor frame for $\mathcal{HS}$. Then the canonical dual atom $\Tilde{S}$ is in $\mathcal{M}^1$.
\end{proposition}
\begin{proof}
    This follows from the result of \cite{gro04} on the localization of frames, along with the arguments above.
    
\end{proof}

\subsection{Frames for $\mathcal{M}^{p,q}_m$ Spaces}
We now turn our attention to operators in general $\mathcal{M}^p$ spaces, which consist of operators with both Weyl symbol and kernel in $M^p(\mathbb{R}^{2d})$:
\begin{proposition}\label{modspsacediscr}
    Let $S\in\mathcal{M}^1$ be a Gabor $\mathcal{HS}$ frame over the lattice $\Lambda \times M$. Then an operator $T\in\mathcal{M}^{\infty}$ is in $\mathcal{M}^p$ if and only if $\{Q_S T(\lambda,\mu)\}_{\lambda,\mu} \in \ell^p(\Lambda \times M)$.
\end{proposition}
\begin{proof}
    By the same argument as \cref{framesymbol}, the Weyl symbol $\sigma_S$ is a Gabor frame on $U(\Lambda\times M)$, and by the Weyl correspondence $\sigma_S\in M^1(\mathbb{R}^{2d})$. Hence by the function frame characterisation of modulation spaces, $\sigma_T \in M^p(\mathbb{R}^{2d})$ if and only if $\{\langle \sigma_T, \pi(U(\lambda,\mu))\sigma_S \rangle\}_{\Lambda\times M} \in \ell^p(\Lambda \times M)$. Thus by \cref{tfshiftsymbol}, this implies that $\{Q_S T(\lambda,\mu)\}_{\lambda,\mu} \in \ell^p(\Lambda \times M)$.
    
\end{proof}
Motivated by this, we continue to find a discrete characterisation of the mixed-norm spaces $\mathcal{M}^{p,q}_m$, with the aim of showing that we can reconstruct $T$ from samples of $Q_S T$ on the lattice. Since the change of variables $U$ on phase space is \textit{not} symplectic, we are in the mixed-norm case unable to translate decomposition results for the Weyl symbol or kernel to results for $T$ in a straight forward way. Indeed the permutation of order of integration in defining modulation spaces leads to very different spaces. 

Considering rank-one operators, it may seem like the tensorisation of modulation spaces will give appropriate spaces. Indeed, the Gelfand triple $(\mathcal{M}^1,\mathcal{HS},\mathcal{M}^{\infty})$ corresponds to projective tensor product, tensor product of Hilbert spaces, and injective tensor product. However, the completion of the algebraic tensor product for arbitrary $p,q$ is not so straightforward.
Instead we will consider the reproducing formula for $Q_S T$ and show that the twisted convolution property along with Young's inequality gives the appropriate discretisation properties. We begin by recalling the following result:
\begin{lemma}
    Given $S\in\mathcal{M}^1_v$, then $Q_{S_0} S\in W(L^{\infty},\ell^1_v)$.
\end{lemma}
This follows from the fact that in the $\mathcal{M}^p_v$ case the Weyl symbols are in $M^p_{\Tilde{v}}(\mathbb{R}^{2d})$, where $\Tilde{v} = v\circ U$.

We consider first the analysis operator:
\begin{proposition}
    Given $S\in\mathcal{M}^1_v$ and lattice $\Lambda \times M$, the analysis map $C_S$, defined as
    \begin{align*}
        C_S (T) = \{Q_S T(\lambda,\mu)\}_{\Lambda \times M},
    \end{align*}
    is a bounded map from $\mathcal{M}^{p,q}_m$ to $\ell^{p,q}_m(\Lambda \times M)$.
\end{proposition}
\begin{proof}
    Combining the twisted convolution identity \cref{twistedrelation} with \cref{weineramalgconv} concerning Wiener-amalgam space convolution relations, we find:
    \begin{align*}
        \|C_S (T)\|_{\ell^{p,q}_m} &= \|Q_S T |_{\Lambda\times M}\|_{\ell^{p,q}_m} \\
        &\leq C \|Q_S T\|_{W(L^{\infty},\ell^1_v)} \\
        &\leq \|Q_{S_0} S\|_{W(L^{\infty},\ell^1_v)} \|T\|_{\mathcal{M}^{p,q}_m}.
    \end{align*}
    
\end{proof}
On the other hand the synthesis operator is similarly bounded:
\begin{proposition}
    Given $S\in\mathcal{M}^1_v$ and lattice $\Lambda \times M$, the synthesis map
    \begin{align*}
        D_S (a) = \sum_{\Lambda \times M} a_{\lambda,\mu} \gamma_{\lambda,\mu}(S),
    \end{align*}
    is a bounded map from $\ell^{p,q}_m(\Lambda \times M)$ to $\mathcal{M}^{p,q}_m$, where the sum converges unconditionally in the case $p,q<\infty$, otherwise in the weak-$*$ topology.
\end{proposition}
\begin{proof}
    A direct calculation shows
    \begin{align*}
        Q_S (D_S(a))(w,z) = \sum_{\Lambda\times M} \langle a_{\lambda,\mu}\gamma_{\lambda,\mu}(S),\gamma_{w,z}(S)\rangle.
    \end{align*}
    Let us denote $K_{\lambda,\mu} = [0,1]^{4d} + (\lambda,\mu)$, the unit square translated by the lattice point $\lambda,\mu$. We then consider the sequence
    \begin{align*}
        \{\Tilde{S}_{\lambda,\mu}\}_{\Lambda\times M} := \Big\{\mathrm{max} \{|Q_S S(w,z)| \, :\,  (w,z)\in K_{\lambda,\mu} \} \Big\}_{\Lambda\times M},
    \end{align*}
    which is well defined since $Q_S S \in W(L^{\infty},\ell^1_v)$. It then follows that
    \begin{align*}
        \|Q_S (D_S(a))\|_{L^{p,q}} &\leq \sum_{\Lambda\times M} |a_{\lambda,\mu}\langle S, \gamma_{w-\lambda,z-\mu}(S)\rangle| \\
        &\leq C_1 \|a * \Tilde{S}\|_{\ell^{p,q}_m} \\
        &\leq C_2 \|a\|_{\ell^{p,q}_m} \|Q_S S\|_{W(L^{\infty},\ell^1_v)}
    \end{align*}
    where $C_1$ depends on the lattice and weight function $m$, and $C_2$ depends on the weight function $m$.
    
\end{proof}
Finally combining the two we get the boundedness of the frame operator:
\begin{proposition}
    Given $S,T\in\mathcal{M}^1_v$ and lattice $\Lambda \times M$, the frame operator
    \begin{align*}
        E_{S,T} (a) = D_T C_S,
    \end{align*}
    is a bounded map from $\mathcal{M}^{p,q}_m$ to itself.
\end{proposition}
The existence of frames (and dual frames) in $\mathcal{M}^1_v$ again follows from the coincidence of $\mathcal{M}^1_v$ with symbols in $M^1_v(\mathbb{R}^{2d})$. Crucially the above results give us the discrete characterisation of the $\mathcal{M}^{p,q}_m$ spaces which we will often make us of:
\begin{theorem}\label{discmodspacechar}
    An operator $T\in \mathcal{M}^{\infty}$ is in the space $\mathcal{M}^{p,q}_m$ if and only if $C_S (T) \in \ell^{p,q}_m$ for some $S\in\mathcal{M}^1_v$ which generates a Gabor frame. 
\end{theorem}

As in the function case we can consider the frame operator as taking the form
\begin{align*}
    E_{S,T}(\cdot) = \sum_{\Lambda\times M} \langle \cdot, \gamma_{\lambda,\mu}(S)\rangle \gamma_{\lambda,\mu}(T).
\end{align*}
It follows then that the frame operator is invariant under shifts of the type 
\begin{align*}
    E_{S,T} \mapsto \gamma_{\lambda,\mu} E_{S,T} \gamma_{\lambda,\mu}^*,
\end{align*}
from which one can see that given an operator $S\in\mathcal{HS}$ which generates a Gabor frame, the canonical dual frame is given by $\Tilde{S} = E_{S,S}^{-1}(S)$.

Finally the boundedness of the frame operator on all $\mathcal{M}^{p,q}_m$ spaces allows us to extend the reconstruction property of Gabor frames in $\mathcal{HS}$ to all $\mathcal{M}^{p,q}_m$ spaces:
\begin{corollary}\label{modspacereconstr}
    Let $S,\Tilde{S}\in\mathcal{M}^1_v$ generate dual Gabor frames on $\Lambda\times M$, that is to say $E_{\Tilde{S},S} = I_{\mathcal{HS}}$. Then for $T\in\mathcal{M}^{p,q}_m$;
    \begin{align*}
        T = \sum_{\Lambda\times M} Q_S T(\lambda,\mu) \gamma_{\lambda,\mu}(\Tilde{S}),
    \end{align*}
    where the sum converges unconditionally for $p,q<\infty$, otherwise in the weak-$*$ topology.
\end{corollary}

\section{Properties of $\mathcal{M}^{p,q}_m$ Spaces}
Equipped with a frame characterisation for the $\mathcal{M}^{p,q}_m$ spaces, we proceed to investigate the properties of the operators in these classes in this section. In \cite{cordero03}, Schatten properties are presented for $\sigma_S\in M^p(\mathbb{R}^{2d})$ in the case $1\leq p \leq 2$, and for $\sigma_S\in M^{p,p'}(\mathbb{R}^{2d})$ in the case $2 \leq p \leq \infty$ where $\frac{1}{p}+\frac{1}{p'} = 1$, using interpolation between the cases $p=1,2$. The same approach can be used for operators $S$ with $\sigma_S\in M^p(\mathbb{R}^{2d})$, where now $2\leq p\leq \infty$, with the correspondence for Feichtinger operators established in \cite{feich22}.
\begin{proposition}\label{schattenincl}
The following inclusions are densely continuous in the norm topology for $p<\infty$ and in the weak-$*$ topology for $p=\infty$:
\begin{enumerate}
    \item For $1\leq p \leq 2$, $\mathcal{M}^p \subset \mathcal{S}^p$.
    \item For $2\leq p \leq \infty$, $\mathcal{S}^p \subset \mathcal{M}^p$.
\end{enumerate}
\end{proposition}
\begin{proof}
    The first case of inclusion in the proposition is known \cite{cordero03}, we concern ourselves with the latter. We have the endpoints $\mathcal{M}^2 = \mathcal{HS}$ and $\mathcal{M}^{\infty} = B(M^1;M^{\infty})$. In both cases, $i=2,\infty$, the inclusion $\mathcal{S}^{i} \hookrightarrow \mathcal{M}^{i}$ is continuous. Hence using complex interpolation (cf. \cite{Si05}, \cite{Fe83}) with $\theta=1-\frac{2}{p}$
    \begin{align*}
        \mathcal{S}^p &= [\mathcal{HS},\mathcal{K}(L^2)]_{\theta},\\
        M^p &= [M^2,M^{\infty}]_{\theta},
    \end{align*}
the inclusion result follows for $2\leq p\leq \infty$.
The inclusions are dense, in the norm topology for $p<\infty$ and the weak-$*$ topology for $p=\infty$. In the case $1\leq p \leq 2$, this follows from the density of $M^1(\mathbb{R}^d)\subset L^2(\mathbb{R}^d)$. In the $2<p\le\infty$ case, we consider the restriction of the decomposition to finite subsets of the lattice;
\begin{align*}
    T_n = \sum_{(\lambda,\mu)\in\Lambda_n\times\Lambda_n} Q_S T(\lambda,\mu) \pi(\mu)\Tilde{g} \otimes \pi(\lambda) \Tilde{g}
\end{align*}
Such operators are of finite rank operators from $M^{\infty}(\mathbb{R}^d)$ to $M^1(\mathbb{R}^d)$ and hence in every $\mathcal{S}^p$ class, and are dense in $\mathcal{M}^p$. Hence $\mathcal{S}^p$ is dense in $\mathcal{M}^p$.

\end{proof}

In the case of mixed-norm $\mathcal{M}^{p,q}$ spaces, discretisation gives an easy means to show the following:
\begin{proposition}\label{singvalsdecomp}
    Let $T\in\mathcal{M}^{p,q}$. Then there exists a decomposition 
    \begin{align*}
        T = \sum_{n\in\mathbb{N}} s_n \phi_n \otimes \psi_n,
    \end{align*}
    where $\{s_n\}_N \in \ell^q$, $\phi_n \in M^1(\mathbb{R}^{d})$ and $\psi_n \in M^p(\mathbb{R}^{d})$ such that $\|\phi_n\|_{M^1}=1$ and $\|\psi_n\|_{M^p}=1$ for all $n$. The sum is understood to converge unconditionally for $p,q < \infty$, otherwise in weak-$*$ topology. 
\end{proposition}
\begin{proof}
    Let $T\in\mathcal{M}^{p,q}$. Using \cref{modspacereconstr}, and let $S=g\otimes g$ for some $g\in M^1(\mathbb{R}^d)$, such that $g$ generates a Gabor frame on $\Lambda$. We then have
    \begin{align*}
        T = \sum_{\Lambda\times \Lambda} Q_S T(\lambda,\mu) \gamma_{\lambda,\mu}(\Tilde{S})
    \end{align*}
    with $\{Q_S T\}_{\Lambda\times\Lambda}\in \ell^{p,q}$. For a fixed $\mu$, define $\psi_{\mu} := \frac{\sum_{\lambda} Q_S T(\lambda,\mu) \pi(\lambda)\Tilde{g}}{\| Q_S T(\lambda,\mu) \|_{\ell^p}}$, which is a normalised element of $M^p(\mathbb{R}^d)$, by the discrete characterisation of function modulation spaces. Furthermore the sequence $s=\{s_{\mu}\}_{\mu\in \Lambda}$, where for each fixed $\mu$, $s_{\mu} := \| Q_S T(\lambda,\mu) \|_{\ell^p}$, is an element of $\ell^q$ by \cref{discmodspacechar}. Hence the decomposition
    \begin{align*}
        T = \sum_{\mu\in \Lambda} s_{\mu}  \pi(\mu)\Tilde{g}\otimes \psi_{\mu}
    \end{align*}
    satisfies the claim.
    
\end{proof}
\begin{corollary}\label{compactop}
    For any $1 \leq p,q < \infty$, $\mathcal{M}^{p,q}\subset \mathcal{K}(M^{p'};M^q)$.
\end{corollary}
In this sense, the Weyl product relations from \cite{holst07} become evident, wherein the composition of operators $S\in \mathcal{M}^{p,q}$,$T\in\mathcal{M}^{p',q'}$ are contained in another operator modulation class based on $p,q,p',q'$. In particular, for $p>2$, one does not in general have that $\mathcal{M}^p \circ \mathcal{M}^p\subset\mathcal{M}^p$, since $M^{p}(\mathbb{R}^{d})\not\subset M^{p'}(\mathbb{R}^{d})$ and so the codomain of $S \in \mathcal{M}^{p}$ is not necessarily contained in $M^{p'}(\mathbb{R}^{d})$. Hence the composition $T \circ T$ for some $T\in \mathcal{M}^{p}$ is not necessarily bounded on $M^p(\mathbb{R}^{d})$.

One may show the stronger result that $\mathcal{M}^{p,q}$ operators are also $p$-summing. For this, a similar approach to that used for the classical Hille-Tamarkin operators (cf. \cite{labate01}) is used, along with the infinite matrix characterisation of $\mathcal{M}^{p,q}$ in the same manner as in \cref{compactop}:
\begin{proposition}\label{psumming}
    For $1\leq p,q < \infty$, $\mathcal{M}^{p,q}\subset \Pi^q(M^{p'};M^q)$, where $\Pi^p(M^{p'};M^q)$ is the space of $p$-summing operators from $M^{p'}(\mathbb{R}^{d})$ to $M^q(\mathbb{R}^{d})$.
\end{proposition}
\begin{proof}
    Let $\{f_n\}_{n\leq N}$ be a finite sequence of functions in $M^{p'}(\mathbb{R}^{2})$, and $g$ generate a Gabor frame on the lattice $\Lambda$ as above with dual atom $\Tilde{g}$, and $S=g\otimes g$. Denote by $c^n = \{c_{\lambda}^n\}_{\lambda} \in \ell^{p'}$ the coefficients of $f_n$ with respect to the Gabor frame generated by $g$ on $\Lambda$. Again, for any $T\in\mathcal{M}^{p,q}$, $T$ can be expressed as 
    \begin{align*}
        T = \sum_{(\lambda,\mu)\in\Lambda\times\Lambda} Q_S T(\lambda,\mu) \pi(\mu)\Tilde{g} \otimes \pi(\lambda) \Tilde{g}.
    \end{align*}
    For convenience, we use the notation $\|\cdot\|_{\ell^p(\lambda)}$ to denote that taking the $\ell^p$ norm over the sequence indexed by $\lambda$ holding other indices constant. We then see
    \begin{align*}
        \sum_{n\leq N} \|Tf_n\|_{M^q}^q &= \sum_{n\leq N} \big\|\sum_{(\lambda,\mu)\in\Lambda\times\Lambda} Q_S T(\lambda,\mu)  \langle f_n, \pi(\lambda) \Tilde{g}\rangle\pi(\mu)\Tilde{g} \big\|_{M^q}^q \\
        &\asymp \sum_{n\leq N}  \sum_{\mu\in\Lambda} \Big| \sum_{\lambda\in\Lambda} Q_S T(\lambda,\mu) \langle f_n, \pi(\lambda) \Tilde{g}\rangle\Big|^q  \\
        &= \sum_{n\leq N} \sum_{\mu\in\Lambda} \|Q_S T(\lambda',\mu)\|_{\ell^p(\lambda')}^q  \sum_{\lambda\in\Lambda}\Big| \frac{Q_S T(\lambda,\mu)}{\|Q_S T(\lambda',\mu)\|_{\ell^p(\lambda')}} \cdot \langle f_n, \pi(\lambda) \Tilde{g}\rangle\Big|^q.
    \end{align*}
    We define $g_{\mu} := \{\frac{Q_S T(\lambda,\mu)}{\|Q_S T(\lambda',\mu)\|_{\ell^p(\lambda')}}\}\in \ell^p(\lambda)$, ie the normalised rows of the infinite matrix, such that $\|g_{\mu}\|_{\ell^p}=1$. Then 
    \begin{align*}
        \sum_{n\leq N} \|Tf_n\|_{M^p}^q &\asymp \sum_{\mu\in\Lambda} \|Q_S T(\lambda,\mu)\|_{\ell^p(\lambda)}^q \sum_{n\leq N} \Big| \langle g_{\mu}, c^n\rangle_{\ell^p,\ell^{p'}} \Big|^q \\
        &\lesssim \|T\|_{\mathcal{M}^p}^q \sup_{g\in B(\ell^p)}\sum \Big| \langle g, c^n\rangle_{\ell^p,\ell^{p'}} \Big|^q,
    \end{align*}
as required.

\end{proof}
Since $M^p(\mathbb{R}^{d})$ spaces have the inclusion property $M^p(\mathbb{R}^{d})\subset M^q(\mathbb{R}^{d})$ for $p<q$, with $\|\cdot\|_{M^q}\leq \|\cdot\|_{M^p}$, for $p\leq 2$ we can consider either the restriction of an operator $T\in\mathcal{M}^p$ as a bounded operator $T\in\mathcal{L}(M^p;M^p)$, or as a bounded operator $T\in\mathcal{L}(M^{p'};M^{p'})$. Along with Section 2.b of \cite{konig}, which states that a p-summing operator from a Banach space $X$ into itself has $r$-summable eigenvalues, where $r=\max\{2,p\}$, this gives Theorem 3.1 of \cite{labate01} as a corollary:
\begin{corollary}[Theorem 3.1, \cite{labate01}]
    Let $1 \leq p \leq 2$. Then for any $T\in\mathcal{M}^p$:
    \begin{itemize}
        \item $T$ is a compact operator from $M^p(\mathbb{R}^{d})$ to $M^p(\mathbb{R}^{d})$ with 2-summable eigenvalues.
        \item $T$ is a compact operator from $M^{p'}(\mathbb{R}^{d})$ to $M^{p'}(\mathbb{R}^{d})$ with $p$-summable eigenvalues.
    \end{itemize}
\end{corollary}

We now consider how the $\mathcal{M}^{p,q}$ relate to operators with Weyl symbol in some modulation space $M^{p,q}(\mathbb{R}^{2d})$. To that end, we find from \cref{singvalsdecomp} the following:
\begin{corollary}\label{weylsymbincla}
    If $T\in\mathcal{M}^{p,1}$, then $\sigma_T \in M^{1,p}(\mathbb{R}^{2d})$.
\end{corollary}
\begin{proof}
    If $T\in\mathcal{M}^{p,1}$, then $T$ admits a decomposition 
    \begin{align*}
        T = \sum_{n\in \mathbb{N}} s_n \phi_n \otimes \psi_n
    \end{align*}
    with $s_n,\phi_n,\psi_n$ as in \cref{singvalsdecomp} for $q=1$. Letting $S = \varphi_0 \otimes \varphi_0 $, it follows that
    \begin{align*}
        \|\sigma_T\|_{M^{1,p}} &= \|Q_S T(U^{-1}(w,z))\|_{L^{1,p}} \\
        &= \Big\| \sum_{n\in \mathbb{N}} s_n V_{\varphi} \phi_n\Big(w+\frac{Jz}{2}\Big)\overline{V_{\varphi} \psi_n\Big(w-\frac{Jz}{2}\Big)} \Big\|_{L^{1,p}} \\
        &\leq \sum_{n\in \mathbb{N}} |s_n|\cdot \Big\| V_{\varphi} \phi_n\Big(w+\frac{Jz}{2}\Big)\overline{V_{\varphi} \psi_n\Big(w-\frac{Jz}{2}\Big)} \Big\|_{L^{1,p}} \\
        &= \sum_{n\in \mathbb{N}} |s_n|\cdot \| |V_{\varphi} \phi_n| * |\Check{V_{\varphi} \psi_n}|\|_{L^{p}} \\
        &\leq \sum_{\mu\in M} |s_n|\cdot \| V_{\varphi} \phi_n \|_{L^1} \|\Check{V_{\varphi} \psi_n}\|_{L^{p}} \\
        &= \sum_{\mu\in M} |s_n|
    \end{align*}
    where we have used \cref{cohensstftform}, along with the change of variables $w \mapsto w - \frac{Jz}{2}$ and Young's inequality.

\end{proof}

\begin{corollary}\label{weylsymbinclb}
    If $\sigma_T \in M^{\infty,q}(\mathbb{R}^{2d})$, then $T\in\mathcal{M}^{q,\infty}$.
\end{corollary}
\begin{proof}
    This follows from \cref{weylsymbincla} and a duality argument.
    
\end{proof}
By considering the two inclusion relations \cref{weylsymbincla} and \cref{weylsymbinclb}, and recalling that $\mathcal{M}^p$ consists of operators with Weyl symbol in $M^p(\mathbb{R}^{2d})$, we can use an interpolation argument to extend the previous results:
\begin{theorem}
    Given $1\leq p \leq q \leq \infty$, if $T\in\mathcal{M}^{q,p}$ then $\sigma_T \in M^{p,q}(\mathbb{R}^{2d})$. Conversely given \newline $1 \leq q \leq p \leq \infty$, if $\sigma_T \in M^{p,q}(\mathbb{R}^{2d})$ then $T\in\mathcal{M}^{q,p}$.
\end{theorem}
\begin{proof}
    Denote by $\Tilde{\mathcal{M}}^{p,q}$ the space of operators with Weyl symbol in $M^{p,q}(\mathbb{R}^{2d})$ with corresponding norm. Clearly in the case $q=p$, the spaces $\mathcal{M}^p$ and $\Tilde{\mathcal{M}}^p$ coincide, and hence $\mathcal{M}^p \hookrightarrow \Tilde{\mathcal{M}}^p$ is a continuous inclusion. In the case $p=1$, the inclusion $\mathcal{M}^{1,q} \hookrightarrow \Tilde{\mathcal{M}}^{1,q}$ is a continuous inclusion by \cref{weylsymbincla}. We recall that $\mathcal{M}^{p,q}$ spaces can be defined as a (vector--valued) Wiener-amalgam space as in \cref{wieneramalg}, and that for the vector--valued Lesbegue spaces $L^q(\mathbb{R}^d;M^p(\mathbb{R}^d))$, we have the same complex interpolation spaces as the scalar case (Theorem 5.1.2, \cite{BeLo76}). It follows then from the same argument as the classical Wiener amalgam case \cite{Fe81} that the (complex) interpolation spaces with $\theta = (1-p)(1-\tfrac{1}{q})^{-1}$, we have
    \begin{align*}
        \mathcal{M}^{q,p} &= [\mathcal{M}^{q,1},\mathcal{M}^{q,q}]_{\theta}.
    \end{align*}
    Using the classical interpolation of function modulation spaces on the Weyl symbols, we also have
    \begin{align*}
        \Tilde{\mathcal{M}}^{p,q} &= [\Tilde{\mathcal{M}}^{1,q},\Tilde{\mathcal{M}}^{q,q}]_{\theta}
    \end{align*}
    Hence the inclusion result follows by interpolation for $1\leq p \leq q$. The converse argument for $q\leq p$ follows in the same manner now using \cref{weylsymbinclb}.
    
\end{proof}
Combining this with \cref{psumming} gives:
\begin{corollary}
    Given $p\leq q$, if $\sigma_T \in M^{q,p}(\mathbb{R}^{2d})$ then $T \in \Pi^q(M^{p^\prime};M^q)$.
\end{corollary}

In \cite{cordero03}, the authors make use of the above relations by identifying a (pure) localisation operator as the convolution of a mask distribution with a Wigner distribution, which is of course precisely the Weyl symbol of a rank-one operator. Since the space $\mathcal{M}^1$ is precisely the space of operators with Weyl symbols in $M^1(\mathbb{R}^{2d})$, the operator-function convolutions allow us to generalise the classification to mixed-state localisation operators:
\begin{theorem}
    For any $S\in\mathcal{M}^1$, and $a\in M^{\infty}(\mathbb{R}^{2d})$, the mixed-state localisation operator 
    \begin{align*}
        A = a \star S
    \end{align*}
    is bounded on all $M^{p,q}(\mathbb{R}^d)$ spaces.
\end{theorem}
\begin{proof}
    This follows from the convolution relations for modulation spaces and the celebrated result of \cite{groch06} that operators with Weyl symbol in $M^{\infty,1}(\mathbb{R}^{2d})$ are bounded on all $M^{p,q}(\mathbb{R}^d)$ spaces.
    
\end{proof}
Let us close this section with a remark concerning the relation between the coorbit spaces of operators introduced in \cite{dorf22} and the $\mathcal{M}^{p,q}$-spaces introduced in this work. 
 Recalling that the operator STFT \cite{dorf22} satisfies the identity
\begin{align*}
    Q_S T(w,z) = e^{i\pi w_1\cdot w_2}\mathcal{F}_W(\mathfrak{V}_S T(z) )(-w),
\end{align*}
the spaces $\mathfrak{M}^p$ introduced in that work are the $\mathcal{M}^{2,q}$ spaces. This can be seen by recalling that the Fourier-Wigner transform is an isometric isomorphism from Hilbert-Schmidt operators to $L^2(\mathbb{R}^{2d})$.

\bibliographystyle{abbrv}
\bibliography{refs}

\Addresses

\end{document}